\documentclass[11pt]{amsart} 

\usepackage[utf8]{inputenc} \usepackage[left=2.65cm,right=2.65cm,top=2.7cm,bottom=2.7cm]{geometry}
\usepackage{graphicx} 
\usepackage{color}

\makeatletter
\renewcommand\subsubsection{\@secnumfont}{\bfseries}%
\renewcommand\subsubsection{\@startsection{subsubsection}{3}
  \z@{.5\linespacing\@plus.7\linespacing}{-.5em}%
  {\normalfont\bfseries}}
  
  \makeatother

\usepackage{array} 
\usepackage{paralist} 
\usepackage{verbatim} 
\usepackage{subfig} 

\usepackage{amsfonts}
\usepackage{amsmath}
\usepackage{amsthm}
\usepackage{amssymb}
\usepackage{mathtools}

\newcommand{\mel}{\MoveEqLeft}
\usepackage{amsthm}

\numberwithin{equation}{section}

\newtheorem{theorem}{Theorem}[section]

\newtheorem{definition}[theorem]{Definition}
\newtheorem{infotheorem}[theorem]{Very informal theorem}

\newtheorem{example*}{Example\textsuperscript{*}}
\newtheorem{proposition*}{Proposition\textsuperscript{*}}
\newtheorem{corollary}[theorem]{Corollary}
\newtheorem{corollary*}{Corollary\textsuperscript{*}}
\newtheorem{proposition}[theorem]{Proposition}
\newtheorem{assumption}[theorem]{Assumption}
\newtheorem{lemma}[theorem]{Lemma}

\newtheorem{convention}[theorem]{Convention}

\theoremstyle{definition}
\newtheorem{example}[theorem]{Example}

\newtheorem{temporaryassumption}[theorem]{Temporary assumption}

\def\star{\textsuperscript{*} }

\def\Limes#1#2 {\lim\limits_{#1\rightarrow #2}}

\def\eps{\epsilon}

\DeclareMathOperator{\dist}{dist}
\def\R{\mathbb{R}}

\def\Q{\mathbb{Q}}

\def\sym{\mathcal{E}}
\def\N{\mathbb{N}}

\def\XXint#1#2#3{{\setbox0=\hbox{$#1{#2#3}{\int}$ }
\vcenter{\hbox{$#2#3$ }}\kern-.59\wd0}}

\DeclareMathOperator{\diam}{Diam}
\DeclareMathOperator{\dom}{Dom}
\DeclareMathOperator{\ran}{R}

\DeclareMathOperator{\supp}{supp}
\DeclareMathOperator{\BD}{BD}

\DeclareMathOperator{\BVA}{BV^\mathbb{A}}

\DeclareMathOperator{\BV}{BV}

\DeclareMathOperator{\curl}{curl}

\def\norm#1{\left\lVert #1 \right\rVert}
\def\scalar#1#2{\langle #1,#2 \rangle}
\def\de{\partial}
\renewcommand{\div}{\operatorname{div}}

\def\dd{\,\mathrm{d}}
\def\dx{\,\mathrm{d}x}

\def\dy{\mathrm{d}y}
\def\dt{\mathrm{d}t}

\newcommand{\mres}{\mathbin{\vrule height 1.6ex depth 0pt width
0.13ex\vrule height 0.13ex depth 0pt width 1.3ex}}

\mathtoolsset{showonlyrefs}

\usepackage{hyperref}

\title{Total $\mathbb{A}$-variation flows}

\begin{document}
\thanks{The original version of this work was written while the author
was funded by the Deutsche Forschungsgemeinschaft (DFG, German Research Foundation) under Germany's Excellence Strategy EXC 2044 --390685587, Mathematics M\"unster: Dynamics--Geometry--Structure. A revision was written while the author received funding from the European Research Council (ERC) under the European Union's Horizon 2020 research and innovation programme through the grant agreement~862342
\copyright 2025 by the author.
}
\date{\today}

\author{David Meyer}
\address{Instituto de Ciencias Matem\'aticas, Consejo Superior de Investigaciones Cient\'\i ficas, 28049 Madrid, Spain}
\email{david.meyer@icmat.es}


\maketitle
\begin{abstract}We study the $L^2$-gradient flows, $\de_t u-\div(\mathrm{D}f(x,\mathbb{A}u))=0$, of functionals of the type $\int_{\Omega}f(x,\mathbb{A}u)\dx$, where $f$ is a convex function of linear growth and $\mathbb{A}$ is some first-order linear constant-coefficient differential operator.

To this end, we identify the relaxation of the functional to the space $\BVA\cap L^2$, identify its subdifferential, and show pointwise representation formulas for the relaxation and the subdifferential, both with and without Dirichlet boundary conditions. The existence and uniqueness then follow from abstract semigroup theory.

We further show that our solutions can be obtained as limits of the corresponding flows with $p$-growth as $p\searrow 1$.
\end{abstract}

\section{Introduction}

Let $\Omega\subset\R^n$ be some bounded open domain with Lipschitz boundary. We are concerned with functionals of the type \begin{align}
F(u):=\int_\Omega f(x,\mathbb{A}u(x))\dx.\label{int F}
\end{align}
and the corresponding gradient flows, given, at least formally, by \begin{align}
&\de_t u=\div(A^T\de_y f(x,\mathbb{A}u))\quad\text{ in $\Omega$}\label{main eq}
\end{align} 
with either Neumann or Dirichlet boundary conditions specified further below. 

Here $f(x,y):\overline{\Omega}\times \R^{m\times n}\rightarrow\R$ is assumed to be convex with linear growth in the second variable and to be measurable in the joint variable. $\Omega\subset \R^m$ is assumed to be a bounded open set with Lipschitz boundary. $\mathbb{A}=A\text{D}$ is a constant coefficient linear differential operator of first order, acting on functions $u:\Omega\rightarrow \R^n$ such as for instance the full gradient $\mathbb{A}=\text{D}$, the symmetric gradient $\sym u:=\frac{\text{D}u+(\text{D}u)^T}{2}$ or the divergence $\mathbb{A}=\div$.  

In the light of the possible lack of regularity of $f$ and its linear growth, there are multiple difficulties in showing well-posedness of \eqref{main eq} or even obtaining a reasonable definition of solution of \eqref{main eq}.

The first difficulty is that $f$ might not be differentiable and hence the expression on the left-hand side in \eqref{main eq} is not defined even for smooth $u$. There is a natural workaround for this, which consists of working with the subdifferential of the functional \eqref{int F}. Classical semigroup theory (see e.g.\ \cite[Chapter 3]{MR0348562}) then shows that if the functional $F$ can be interpreted as a convex and lower semicontinuous functional on a Hilbert space (i.e.\ $L^2$), there is a unique solution to the intital value problem $\de_t u(t)+\de F(u(t))\ni 0$ for every initital datum in the Hilbert space, where $\de F$ denotes the subdifferential.

It is of course a priori not clear how to suitably interpret $F$ as a convex and lower-semicontinuous functional on $L^2$ and to what extent its subdifferential agrees with \linebreak$-\div(A^T\de_y f(x,\mathbb{A}u))$. Coming up with a suitable definition of the functional and a characterisation of the subdifferential are the goals of this paper.

The difficulty in doing so is that if one wants to work with a convex and lower-semicontinuous version of \eqref{int F}, one is forced to work with functions for which $\mathbb{A}u$ is merely a bounded Radon measure, as $L^1$ is not closed under weak\star convergence. Furthermore, as a consequence of the exponent $1$, Korn-type inequalities are not true for these types of spaces and they only agree with $\BV$ if $\mathbb{A}$ trivially controls the full gradient (see e.g.\ \cite{ornstein1962non,Kirchheim}). Hence, the natural energy space to work with is the space of functions with bounded $\mathbb{A}$-variation, defined as  \begin{align}
\BVA:=\left\{u\in L^1(\Omega)\,\big|\,\mathbb{A}u\in \mathcal{M}(\Omega,\R^{m\times n})\right\},
\end{align} 
where $\mathcal{M}(\Omega,\R^{m\times n})$ denotes the space of $\R^{m\times n}$-valued Radon measures.

If $\mathbb{A}u$ is not a function, one needs a suitable reinterpretation of $F$. There are multiple classical ways of applying convex functions with linear growth to measures, the first being the Serrin-Goffman type extension (see e.g.\ \cite{Giaquinta}) \begin{align}
F_{SG}(u)=\int_\Omega f(x,(\mathbb{A}u)^a(x))\dx+\int_\Omega f^\infty(x,\frac{(\mathbb{A}u)^s}{|\mathbb{A}u|^s}(x))\dd |\mathbb{A}u|^s(x)\label{gs}
\end{align}
where $(\mathbb{A}u)^s$ and $(\mathbb{A}u)^a$ denote the singular and the absolutely continuous part w.r.t.\ to the Lebesgue measure of the measure $\mathbb{A}u$, and we identify the absolutely continuous part with its density. In full generality, if the dependence of $f$ on $x$ is very rough, such a functional might behave very badly, even in one dimension and for the full gradient, for instance, it might not be lower-semicontinuous (see e.g.\ \cite[p.\ 513]{Bouchitte}, \cite[Section 5]{Valadier}).

Another classical way (see e.g.\ \cite{demengel1984convex}) of defining convex functions of a measure is to use convex duality theory, i.e.\ to set \begin{align}
F_{dual}(u)=\sup_{z}\int_\Omega -f^*(x,z)+z:\mathbb{A}u\dx,\label{dual f}
\end{align}
for $z$ varying over a suitable class for which the product $z:\mathbb{A}u$ makes sense and for which $f^*(x,z)\in L^1$. Such a functional, as a supremum of affine functionals, is trivially convex and lower semicontinuous. If everything behaves nice enough to exchange the supremum and the integral, then for $u$ for which $\mathbb{A}u$ is an $L^1$ function, then this functional agrees with $\int f(x,\mathbb{A}u)\dx$, but in full generality, the constraint $f^*(x,z)\in L^1$ can prevent this for rough $f$ even in one dimension (see e.g.\ \cite{Valadier}).

The third possibility is to simply consider the abstract lower-semicontinuous hull of $F$, i.e. to look at \begin{align}
F_{relaxed}(u)=\liminf_{\substack{u_l\rightarrow u\\ \mathbb{A}u_l\in L^1}}\int_\Omega f(x,\mathbb{A}u_l(x))\dx.
\end{align}
Our first result is roughly the following:

\begin{infotheorem}
The functionals $F_{relaxed}$ and $F_{dual}$ agree with no regularity assumption on $f$ (besides measureability in the joint variable and convexity in the second variable). 
\end{infotheorem}

Rigorous statements can be found in Sections \ref{section 4} and \ref{section 6}. We remark that in the case $\mathbb{A}=\mathrm{D}$ similar statements are known, see e.g.\ \cite[equation (4.19)]{Bouchitte}.

It can be shown almost exactly as in the full gradient case that under some additional continuity assumptions on $f$, both functionals also agree with $F_{SG}$, see Proposition \ref{resh1} and \ref{int rep cont} below.

Our characterization of the subdifferential of $F$ (extended to $L^2$ as $+\infty$ for $u\in L^2\backslash \BVA$) is then roughly the following: \begin{infotheorem} 
$v\in L^2$ lies in $\de F_{dual}(u)$ if and only if it is of the form $-\div z$ for a $z$ for which the supremum in \eqref{dual f} is attained. 

Under additional continuity assumptions on $f$, this also implies that such a $z$ must lie in the pointwise subdifferential $\de_y f(x,\mathbb{A}^au(x))$ a.e.
\end{infotheorem}

Rigorous versions (which also contain boundary conditions and a description of the behaviour of $z$ on the singular part of $\mathbb{A}^su$) can be found in Sections \ref{section 4} and \ref{section 6}.

In particular, well-posedness of \eqref{main eq} then follows directly from abstract semigroup theory.

\begin{corollary}\label{coro intro}
There is a unique solution to \eqref{main eq} for initial data in $L^2(\Omega,\R^n)$, which can be described with the subdifferential characterisation above.
\end{corollary}

The validity of such a theory hinges on being able to use a sufficiently wide class of $z$ in the duality formulation \ref{dual f}. The class of continuous functions is not appropriate here, since there is no reason why $\de_y f(x,\mathbb{A}^au(x))$ should be continuous. It is however natural to consider $z$ which have a (distributional) divergence in $L^2$, which is enough to give a meaning to ``$\int z:\mathbb{A}u$'' by partial integration. We will therefore develop a suitable theory for this pairing in Section \ref{section 3}.\smallskip

Let us also say a few more words about the boundary conditions: There are two main possibilities, the first one being a Neumann-type condition, which reads as $\scalar{z}{\nu}=0$, where $\nu$ is the normal and the scalar product is taken row-wise. This corresponds to $F$ as it is written above.

The case of a Dirichlet boundary condition is more involved. Formally the Dirichlet boundary condition $u\big|_{\de\Omega}=u_1$ corresponds to restricting the functional $F$ to $\mathrm{BV}_{u_1}^{\mathbb{A}}$ (again extended as $+\infty$ to the rest of $L^2$). On the one hand, this requires a theory of traces for the space $\BVA$. It has been shown in \cite{Diening} that the existence of a trace in $L^1$ is equivalent to a certain ellipticity condition for $\mathbb{A}$, the so-called $\mathbb{C}$-ellipticity. On the other hand, the space $\mathrm{BV}_{u_1}^\mathbb{A}$ is not going to be closed under weak\star convergence even if a trace exists, even if $\mathbb{A}=\mathrm{D}$, and there are many classical examples in which the corresponding elliptic problems are impossible to solve, see e.g.\ \cite[Section 2.3]{gorny2024functions}, \cite{finn1965remarks}. The typical workaround for this, which we will also follow, is to consider the lower semicontinuous hull of the functional restricted to $\BV_{u_1}^{\mathbb{A}}$. Following the classical theory, this corresponds, if $f$ fulfills some continuity assumptions, to adding a penalty term $\int_{\de\Omega} f^\infty(x,(u_1-u)\otimes \nu)\dd\mathcal{H}^{n-1}(x)$ to the functional \eqref{gs}. We will see in Theorem \ref{T relax 2} that, if $\mathbb{A}$ is $\mathbb{C}$-elliptic, the relaxed functional can still be described through the duality approach and that there is a similar characterization of the subdifferential in theorem \ref{main thm 2} with the boundary condition $\scalar{z}{\nu}\in \de_y f^\infty(x,(u_1-u)\otimes \nu)$ on $\de\Omega$.

These results also additionally provide a characterisation of minimisers of $F$ (resp.\ the relaxed functional $F_{dual}$) as for those $0$ must trivially be in the subdifferential, which is useful e.g.\ for studying the properties of minimisers, see e.g.\ \cite{beck2015convex,meyer2025attainment,moradifam2018existence}

Another possible way of making sense of the solution to \eqref{main eq} is to consider it as a (at least formal) limit of more regular problems, such as for instance $\de_t-\div(A^T\de_y f^q(x,\mathbb{A}u))=0$ for $q\searrow 1$, which has a much more reasonable concept of solution as one can work with functions for which $\mathbb{A}u\in L^q$ instead of using measures. In Section \ref{section 7}, we will show that convergence of these solutions does indeed hold as $q\searrow 1$ and that one obtains the solutions from Corollary \ref{coro intro} in the limit. Similar results for the full gradient case $\mathbb{A}=\mathrm{D}$ were shown in \cite{tolle2011convergence,MR4030475,Schtzler}.

\subsubsection{Existing literature}
Such problems have a long history and many specific cases have already been studied before, the first, to the best of the authors knowledge, being the parabolic version of the (nonparametric) minimal surface problem, that is $\de_t u-\div\frac{\mathrm{D} u}{\sqrt{1+|\mathrm{D} u|^2}}$=0, considered in \cite{lichnewsky1978pseudosolutions}, and which corresponds to $F=\int \sqrt{1+|\mathrm{D}u|^2}\dx$. 

Another typical model example is the total variation flow $
\de_t u-\div\left(\frac{\text{D}u}{|\text{D}u|} \right)=0$, which is motivated by image restoration models and corresponds to $F=\int_\Omega |\mathrm{D}u|\dx$. A good theory of the concept of solution and the existence and uniqueness was developed in \cite{Nonlinearboundary,andreu2001,bellettini2002total,andreu2001dirichlet,andreu2002some}, see also the references therein. These techniques (which are in a similar spirit to the ones used in this paper) were also generalised to the more general case $F=\int_\Omega f(x,\mathrm{D}u)\dx$ with some continuity assumptions on $f$ in \cite{Caselles,gorny2022duality,Mazon3}. The case of a general rough $f$ has, to the best of the authors' knowledge, only been considered for $1$-homogeneous $f$ in \cite{Moll}.

A different approach based on variational inequalities, which does for instance allow for time-dependent boundary conditions, but gives less ``pointwise'' control of the solutions, has been established in \cite{Bgelein2016,Kinnunen,Schtzler,bogelein2016obstacle,bogelein2015time}.

For $\mathbb{A}\neq \text{D}$ the prototypical case is $\mathbb{A}=\sym$. In this case, the space $\BVA$ is called the space of functions of bounded deformation, abbreviated $\BD$, and was first studied in \cite{suquet1978existence,MR0537013,Temam}. Elliptic problems involving linear growth and the symmetric gradient were first studied in \cite{MR0592713} and come from contexts such as plasticity \cite{suquet1978existence,Temam}, or fluid mechanics in the form of e.g.\ the inviscid Bingham model \cite{bouchut2025h1,bouchut2014convergence} or sea-ice models \cite{brandt2022rigorous,liu2022well,denk2025singular}. 

The flow corresponding to $F(u)=\int_\Omega |\div u|\dx$ has been studied in \cite{briani2011gradient} and the  corresponding function spaces have been investigated mostly in the context of conservation laws in e.g.\ \cite{Chen,chen2003extended}

The study of more general differential operators in the context of linear growth functionals goes back to Van Schaftingen in \cite{vanSchaftingen} and the space $\BVA$ has been further investigated in \cite{Diening, Gmeineder, vanSchaftingen, arroyo2020slicing,MR4340491,MR4029794}.
In the context of image restoration, integrands of the type $f(x,\mathbb{A}u)$ with very mild continuity assumptions and general $\mathbb{A}$ were considered in e.g.\ in \cite{MR4492883,d2025relaxation}. 
 
The regularity of such elliptic problems has been considered e.g.\ in \cite{gmeineder2019sobolev,MR4110431,MR4188326,MR4571629} and is an interesting problem for the future in the parabolic case. 




\subsection{Structure of the paper} In Section \ref{section 2}, we will introduce the space $\BVA$ and some of its properties. Furthermore, we will cover the preliminaries from convex analysis. In Section \ref{section 3}, we will introduce a duality pairing for $\mathbb{A}u$ and study Green's-type formulas for $\BVA$. In Section \ref{section 4}, we will introduce the functional $F$ without boundary terms and the precise formulation of the theorem about the subdifferential. In Section \ref{section 6}, we will study the functional with Dirichlet boundary conditions. In Section \ref{section 7}, we will prove that the flow can be approximated by the flow for $f^q$, as discussed above.
The proofs of the subdifferential characterisations and the relaxation are contained in Section \ref{proof section}.

\section{Preliminaries and Notation }\label{section 2}
\subsubsection{General notation}
The divergence of matrix-valued functions is always taken row-wise i.e. \begin{align}
(\div g)_j=\sum_{i}\de_i g_{ji}.
\end{align} 
Let $\Omega\subset \R^n$ be open and bounded with Lipschitz boundary throughout this entire work. Let $\nu$ always denote the outer unit normal on $\de\Omega$. Let $\mathcal{M}(\Omega,\R^m)$ denote the space of bounded (signed) $\R^m$-valued Radon measures on $\Omega$. For a measure $\mu\in \mathcal{M}(\Omega,\R^m)$ we let $\norm{\mu}:=|\mu|(\Omega)$ denote its total variation norm. For such a measure, we let $\mu^a$ and $\mu^s$ be its absolutely continuous and singular parts with respect to the Lebesgue measure. 
We will follow the convention of identifying absolutely continuous (w.r.t.\ the Lebesgue measure) measures with their densities.
Norms of $\R^l$-valued functions are always taken with respect to the Euclidean norm on $\R^l$.

We always let $C$ denote a positive constant that depends on $n,m, \Omega$ and $A$, but not on the other variables unless explicitly specified and is allowed to change values from line to line. Sometimes, when there is such a constant, we also write $a\lesssim b$ instead of $a\leq Cb$.

\subsection{The space $\BVA$}
\begin{definition}
Let $A\in \R^{(m\times n)\times k}$ denote a (fixed) matrix.
We define the space $\BVA(\Omega)$ of functions of bounded $\mathbb{A}$-variation as the space of $u\in L^1(\Omega,\R^m)$ for which the distribution $$\mathbb{A}u:=A\mathrm{D}u$$ is represented by integration against a bounded ($\R^{k}$-valued) Radon measure, which is also denoted by $\mathbb{A}u$.
We equip it with the norm \begin{align}\norm{u}_{\BVA(\Omega)}:=\norm{u}_{L^1(\Omega,\R^m)}+\norm{\mathbb{A}u}.\end{align}
$\norm{\mathbb{A}u}$ is called the total $\mathbb{A}$-variation of $u$. We define the corresponding Sobolev-type spaces as \begin{align}W^{\mathbb{A},p}(\Omega):=\left\{u\in \BVA(\Omega)\mid\mathbb{A}u\in L^p(\Omega,\R^{k})\right\}.\end{align}
\end{definition}

We note that this space does not depend on the matrix $A$, but only its image, and it is customary to make some non-restrictive structural assumptions on $A$.

\begin{convention}\label{bas ass}
$k=m\times n$ and $A$ is symmetric with $A^2=A$.
\end{convention}

This is \textbf{not restrictive} at all: If $k\neq m\times n$ we can either add dummy entries to $A$ if $k<m\times n$ or remove a part of the complement of the image of $A$ if $k>m\times n$. If $k=m\times n$ we can find an invertible $B\in \R^{(m\times n)\times (m\times n)}$ such that $BA$ is an orthogonal projection and replacing $A$ by $BA$ does not change the function space. We remark that for the later study of the functional $\int f(x,\mathbb{A}u)$ or the associated flow, this is not restrictive either, since one can compose $f$ with $B^{-1}$.\smallskip

This covers for instance the classical space $\BV$ for $\mathbb{A}=\text{D}$, the space $\BD$ for $\mathbb{A} u=\sym u:=\frac{\text{D}u+(\text{D}u)^T}{2}$ if $n=m$ or the space $\mathcal{D}\mathcal{M}^1$ of divergence vector fields for $\mathbb{A}=\div$ (where the in the case of the divergence, one adds dummy entries to fit into the convention \ref{bas ass} as explained above).

By duality, the total $\mathbb{A}$-variation can also be expressed as 

\begin{align}
\norm{\mathbb{A}u}=\sup_{\phi\in C_c^\infty(\Omega, \R^{m\times n}), A\phi=\phi, \norm{\phi}_{\sup}\leq 1}\int_\Omega \scalar{u}{\div\phi} \dx
\end{align}
and a function $u\in L^1$ is in $\BVA$ if and only if this is finite. In particular, this implies that $\norm{\mathbb{A}u}$ is lower semicontinuous w.r.t\ distributional convergence.

We always have $\BV\hookrightarrow \BVA$, the reverse embedding holds only in the trivial case when $A$ has full rank \cite{ornstein1962non}, see also the more modern proof in \cite{Kirchheim}.\medskip

We can not expect that smooth functions are dense in the norm topology of $\BVA$, because for any limit of smooth functions $\mathbb{A}u$ would be absolutely continuous, hence we use a weaker topology: 

\begin{definition}
We say $(u_l)_{l\in \N}$ converges strictly to $u$ in $\BVA$ if $u_l\rightarrow u$ in $L^1$ and $\norm{\mathbb{A}u_l}\rightarrow \norm{\mathbb{A}u}$. For $1\leq p<\infty$ we say that $u_l$ converges strictly to $u$ in $\BVA\cap L^p$ if additionally $u_l\rightarrow u$ in $L^p(\Omega,\R^m)$. We say that  $u_l$ converges strictly to $u$ in $\BVA\cap L^\infty$ if additionally $u_l\overset{\ast}{\rightharpoonup} u$ in $L^\infty(\Omega,\R^m)$.
\end{definition}

If $\mathbb{A}=\text{D}$, this is consistent with the usual definition of strict convergence for $\BV$ (see \cite{Ambrosio}).

\begin{proposition}\label{prop:strict}
$C^\infty(\Omega,\R^m)\cap \BVA(\Omega)\cap L^p(\Omega,\R^m)$ is dense in $\BVA(\Omega)\cap L^p(\Omega,\R^m)$ with respect to strict convergence in $\BVA(\Omega)\cap L^p(\Omega,\R^m)$ for $1\leq p\leq \infty$. 
Furthermore, for each $u\in \BVA(\Omega)\cap L^p(\Omega,\R^m)$ one can find a sequence $u_l\in C^\infty(\Omega,\R^m)\cap\BVA(\Omega)\cap L^p(\Omega,\R^m)$ such that  also $\int_{\Omega}\sqrt{1+|\mathbb{A}u_l^a|^2}\dx\rightarrow\int_{\Omega}\sqrt{1+|(\mathbb{A}u)^a|^2}\dx+\norm{\mathbb{A}u^s}$.

Also $C^\infty(\Omega,\R^m)\cap\BVA(\Omega)\cap L^p(\Omega,\R^m)$ is norm-dense in $W^{\mathbb{A},1}(\Omega)\cap L^p(\Omega,\R^m)$ for $p<\infty$.\end{proposition}
The proposition for $p=1$ can be found in \cite{Diening}, while not explicitly stated, the proof there also works for $p>1$ with very minor modifications.\medskip

For a general $\mathbb{A}$, the space $\BVA$ does not have well-behaved traces, e.g.\ when some derivative is missing in $\mathbb{A}$. For ``elliptic'' $\mathbb{A}$, most of the usual theory still works though. \begin{definition}
The operator $\mathbb{A}$ is called $\mathbb{C}$-elliptic if the linear map \begin{align}v\rightarrow\mathbb{A}[\xi]v:= A\begin{pmatrix} \xi_1v_1& \xi_2v_1&\dots& \xi_nv_1\\ \xi_1v_2& \xi_2v_2&\dots& \xi_nv_2\\\dots\\ \xi_1v_m& \xi_2v_m&\dots& \xi_nv_m\end{pmatrix}\end{align} is  injective (on $\mathbb{C}^m$) for all $\xi\in \mathbb{C}^n\backslash \{0\}$.
\end{definition}

For instance, this is the case if $\mathbb{A}=\mathrm{D}$, if $\mathbb{A}=\sym$ or if $n>2$ for $\mathbb{A}=\sym-\frac{1}{n}\div$ as easy calculations show. Examples of operators for which this does not hold are, for instance, $\mathbb{A}=(\div,\curl)$ or $\mathbb{A}=\sym-\frac{1}{2}\div$ both in $n=m=2$.

\begin{theorem}(\cite{Diening})\phantomsection\label{trace}
If $\mathbb{A}$ is $\mathbb{C}$-elliptic, then there is a continuous (w.r.t.\ to the norm topology) and linear trace $u\rightarrow u|_{\de\Omega}$ from $\BVA(\Omega)$ to $L^1(\de\Omega,\R^m)$, which is the restriction to $\de\Omega$ on the subspace $C^1(\overline{\Omega}, \R^m)$.

 It is in fact already continuous with respect to the strict topology. For any $\phi\in C^1(\overline{\Omega},\R^{m\times n})$ we have a Green's formula \begin{align}
\int_\Omega \scalar{u}{\div (A\phi)}\dx+\int_\Omega \phi:\mathbb{A}u\dx=\int_{\de\Omega}\sum_{i,j}\nu_j(A\phi)_{ij}u_i\dd\mathcal{H}^{n-1}.\label{Green1}
\end{align} Furthermore, for any $u\in \BVA(\Omega)$ the extension by zero $\overline{u}$ is in $\BVA(\Omega')$ if $\Omega'\supset\supset \Omega$ and we have $\norm{\overline{u}}_{\BVA(\Omega')}\lesssim\norm{u}_{\BVA(\Omega)}$ where the implicit constant depends on $\Omega$ (and $\mathbb{A}$).\end{theorem}

The proofs can be found in \cite[Thm.\ 1.2; Lemma 4.14; Thm.\ 4.20; Cor.\ 4.21]{Diening}. We remark that $\mathbb{C}$-ellipticity is also necessary for the existence of a trace in $L^1$ (cf.\ \cite[Thm.\ 4.18]{Diening}.

We will usually omit the ``$|_{\de\Omega}$''.
If $\mathbb{A}$ is $\mathbb{C}$-elliptic and $u_1\in L^1(\de \Omega,\R^m)$ we set \begin{align}
&\BV_{u_1}^\mathbb{A}(\Omega):=\{u\in \BV_{u_1}^\mathbb{A}\,\big|\,u|_{\de \Omega}=u_1\}\\
&W_{u_1}^{\mathbb{A},p}(\Omega):=W^{\mathbb{A},p}\cap \BV_{u_1}^\mathbb{A}.
\end{align}
It is also known that for $\mathbb{C}$-elliptic $\mathbb{A}$  strict approximation with $C^\infty(\overline{\Omega})$-functions is possible \cite[Lemma 4.15]{Diening}, here we need the slightly stronger statement for $\BVA\cap L^p$:

\begin{lemma}\label{strict 2}
If $\mathbb{A}$ is $\mathbb{C}$-elliptic, then $C^\infty(\overline{\Omega},\R^m)$ is strictly dense in $\BVA(\Omega)\cap L^p(\Omega,\R^m)$ for $1\leq p< \infty$. In $W^{\mathbb{A},1}(\Omega)\cap L^p(\Omega,\R^m)$ it is also norm dense.
\end{lemma} 
\begin{proof}
We use that the image of the trace of $W^{1,p}(\R^n\backslash \overline{\Omega},\R^m)$ is dense in $L^1(\de\Omega,\R^m)$ (see \cite[Section 18]{Leoni}).

Let $u_l\in W^{1,p}(\R^n\backslash\Omega,\R^m)$ be such that $\norm{u-u_l}_{L^1(\de\Omega)}< \frac{1}{l}$. Let $\overline{u}_l=\chi_\Omega u+\chi_{\R^n\backslash \Omega}u_l$, then it follows from the Green's formula that  \begin{align}
\mathbb{A}\overline{u}_l=\chi_\Omega \mathbb{A}u+A((u_l-u)\otimes \nu)\mathcal{H}^{n-1}\mres \de\Omega+\chi_{\R^n\backslash \Omega}\mathbb{A}u_l,
\end{align}
which is a (locally) bounded Radon measure.
 We can now take a neighborhood $\Omega_l$ with Lipschitz boundary of $\overline{\Omega}$ such that \begin{align}
|\mathbb{A}\overline{u}_l|(\Omega_l)\leq |\mathbb{A}\overline{u}_l|(\overline{\Omega})+\frac{1}{l}\leq |\mathbb{A}u|(\Omega)+\frac{2}{l},
\end{align}
where we have also used that $A$ has operator norm $\leq 1$ by the Convention \ref{bas ass}.
Then by Proposition \ref{prop:strict} we can strictly approximate $\overline{u}_l|_{\Omega_l}$ in $\BVA(\Omega_l)\cap L^p(\Omega_l,\R^m)$, in particular, we can find a $u_l'\in C^\infty(\Omega_l,\R^m)$ such that $|\mathbb{A}u_l'|(\Omega_l)\leq |\mathbb{A}u|(\Omega)+\frac{3}{l}$ and such that 
$\norm{u-u_l'}_{L^p(\Omega,\R^m)}\leq \frac{1}{l}$. Hence the sequence $u_l'|_\Omega$ converges strictly to $u$ in $\BVA(\Omega)\cap L^p(\Omega,\R^m)$ and lies in $C^\infty(\overline{\Omega},\R^m)$.\smallskip

To show the norm density in $W^{\mathbb{A},1}$, we take a sequence of standard mollifiers $\eta_\eps$, supported in $B_\eps(0)$ and consider the functions $u_{l,\eps}=\overline{u}_l*\eta_\eps$. For every fixed $l$, they converge to $u$ in $L^p(\Omega,\R^m)$. We can estimate \begin{align}
|\mathbb{A}(u-u_{l,\eps})|(\Omega)\leq \norm{(\chi_\Omega\mathbb{A}(u-u_{l,\eps}))*\eta_\eps}+2\norm{\mathbb{A}\overline{u}_l}_{(\Omega+B_\eps(0))\backslash \Omega}.
\end{align}

This goes to zero if we first let $\eps\rightarrow 0$ and then $l\rightarrow \infty$.\end{proof}

\subsection{Preliminaries from convex analysis}
We denote the range and domain of functions with $R$ and $\dom$, with the convention that for functions $g:\mathcal{X}\rightarrow\R\cup\{+\infty\}$ (where $\mathcal{X}$ is an arbitrary set) the value $+\infty$ (and its preimage) are excluded, i.e.\ we set $\dom(g)=\{x\in \mathcal{X}\,|\, g(x)\neq +\infty\}$ and $R(g)=g(\dom(g))$.

Let $H$ be some real separable Hilbert space and $\mathcal{G}:H\rightarrow \R\cup\{+\infty\}$ some convex and lower semicontinuous functional, then for $u,v\in H$ we define \begin{align}
v\in\de\mathcal{G}(u)\iff \scalar{v}{w-u}\leq \mathcal{G}(w)-\mathcal{G}(u)\quad\forall w\in H.
\end{align}
This is called the subdifferential of $\mathcal{G}$. If $H$ is finite-dimensional and $\mathcal{G}<\infty$ everywhere, then for every $u$ the set $\de\mathcal{G}(u)$ is non-empty. If $\mathcal{G}$ is Gateux-differentiable at $u$, then $\de\mathcal{G}(u)=\{\mathrm{D}\mathcal{G}(u)\}$. It will also be convenient to consider $\de\mathcal{G}$ as a relation on $H^2$, i.e.\ by a slight abuse of notation we set \begin{align}
\de\mathcal{G}=\{(u,v)\in H^2\,\big|\,v\in \de\mathcal{G}(u)\}.\label{deg rel}
\end{align}
The subdifferential is monotone, i.e. for $v\in \de\mathcal{G}(u)$ and $v'\in \de\mathcal{G}(u')$ it holds that \begin{align}
\scalar{u-u'}{v-v'}\geq 0.
\end{align}
We will define the range and the domain of the subdifferential as $\dom(\de\mathcal{G})=\{u\in H\,|\, \de\mathcal{G}(u)\neq \emptyset\}$ and $R(\de\mathcal{G})=\bigcup_{u\in\dom(\de\mathcal{G})}\de\mathcal{G}(u)$.

If we take the subdifferential with respect to $1$ variable only, we highlight this by adding the variable in a subscript as in e.g.\ $\de_u\mathcal{G}$. For the sake of better distinction, derivatives will instead be denoted by a $\mathrm{D}$.

\begin{lemma}\phantomsection\label{monotonicity}
If there is at least one $u\in H$ with $\mathcal{G}(u)<\infty$, then for every $\lambda >0$ and $w\in H$ there is exactly one pair $u,v\in H$ with $u+\lambda v=w$ and $v\in \de\mathcal{G}(u)$ (i.e.\ $\de\mathcal{G}$ is a maximal monotone operator).
 In particular, any subset of $\de\mathcal{G}$ which has this property for some $\lambda>0$ must already be the full set $\de\mathcal{G}$.

Furthermore if $w-u\in \de\mathcal{G}(u)$ and $w'-u'\in \de\mathcal{G}(u')$ then $\norm{w-w'}_H\geq \norm{u-u'}_H$.
\end{lemma}

\begin{lemma}\phantomsection\label{subdiff approx}For every $u\in H$  with $\mathcal{G}(u)<\infty$ there is a sequence of $u_l\rightarrow u$ such that $\de\mathcal{G}(u_l)$ is not empty and $\mathcal{G}(u_l)\rightarrow \mathcal{G}(u)$.
\end{lemma}

Let $g:\R^d\rightarrow \R$ be convex, then its Legendre transform is defined as \begin{align}
g^*(y)=\sup_{x}\scalar{x}{y}-g(x).
\end{align}
This function is also convex and lower semicontinuous. The supremum is attained if and only if $y\in \de g(x)$.
 Furthermore it holds that $\overline{\ran(\de g)}=\overline{\dom(g^*)}$.

By Fenchel duality it holds that \begin{align}\label{Fenchel}
g(x)=\sup_y \scalar{x}{y}-g^*(y)
\end{align}
If $g$ is $1$-homogeneous (as it is e.g.\ the case for the recession function), then $g^*$ is $0$ on $\overline{\ran(\de g)}$ and $+\infty$ everywhere else.

If $g:\overline{\Omega}\times \R^l\rightarrow \R$ is convex and such that $\frac{g(x)}{1+|x|}$ is bounded, then the recession function $g^\infty$ is defined as \begin{align}
g^\infty(x,y)=\lim_{x'\rightarrow x,t\rightarrow +\infty}\frac{1}{t}g(x',ty),\label{def rec}
\end{align}
presuming the limit exists.

The main existence result for abstract evolution equations, which we are going to use, is the following classical theorem due to Brezis and Komura.

\begin{proposition}[{\cite[Thm.\ 3.2]{MR0348562}}]\label{brezis}
Let $\mathcal{G}$ be as above. Assume that the set of $u\in H$ for which $\mathcal{G}(u)\neq +\infty$ is dense in $H$. Then for every $u_0\in H$ and every $T>0$ there is a unique solution $u\in C([0,T],H)\cap W_{loc}^{1,\infty}((0,T),H)$ to the following problem \begin{align}
&u(0)=u_0\\
&-\de_t u\in \de \mathcal{G}(u(t))\quad \mathrm{ a.e.}
\end{align}
\end{proposition}

We refer the reader e.g.\ to \cite{Bauschke} for the other proofs and further reading.

\section{Duality pairings for $\BVA$ and Green's formula}\label{section 3}
In general, the product of $\mathbb{A}u$ and an $L^\infty$-function $z$ is not defined, as $z$ is not well-defined on the singular parts of $\mathbb{A}u$, however, if $z$ has a divergence, we will be able to make sense of this by partial integration. In the full gradient case, this construction is due to Anzelotti \cite{Anzellotti1983PairingsBM} and for $\mathbb{A}=\sym$ it was introduced in \cite{KohnTemam}. 

The main idea is that for $\phi\in C_c^\infty(\Omega)$ and sufficiently smooth $u$ and $z$, one has (using the assumed symmetry of $A$ from the Convention \ref{bas ass}) \begin{align}
\int_{\Omega} \phi z:\mathbb{A}u\dx=\int_{\Omega} \phi Az:\mathrm{D}u\dx=-\int_\Omega\sum_{i,j}\de_j\phi (Az)_{ij}u_i+\phi\de_j(Az)_{ij}u_i\dx.
\end{align}
This suggests that the natural space for partial integration is the following:

\begin{definition}
For $1\leq p\leq \infty$ we define \begin{align}X_p(\Omega)=\left\{z\in L^\infty(\Omega,\R^{m\times n})\mid Az=z ,\, \div z\in L^p(\Omega,\R^m)\right\}\end{align} with the norm $\norm{z}_{L^\infty(\Omega,\R^{m\times n})}+\norm{\div z}_{L^p(\Omega,\R^m)}$. Here, the divergence is taken as the distributional one.
\end{definition}

Because we assumed that $A$ is idempotent (cf.\ Convention \ref{bas ass}) and anything orthogonal to the image of $A$ does not affect the product $z:\mathbb{A}u$, the assumption $Az=z$ is natural and not restrictive.

\begin{definition}
Let $u\in \BVA(\Omega)\cap L^p(\Omega,\R^m)$ and let $z\in X_{p'}(\Omega)$ with $\frac{1}{p}+\frac{1}{p'}=1$. Then we define a distribution $(z,\mathbb{A}u)$ in $\Omega$ as \begin{align}
\scalar{(z,\mathbb{A}u)}{\phi}:=-\int_{\Omega}\sum_{i,j}\phi\de_jz_{ij}u_i+\de_j\phi z_{ij}u_i\dx.
\end{align}
\end{definition}

This is well-defined by the integrability assumptions. For $\mathbb{A}=\mathrm{D}$, this is called the Anzelotti pairing, and we will also use that name for our pairing.

We have the following trivial but useful examples:

\begin{example}\phantomsection\label{pairingtrivial}
If $z$ is continuous, then $(z,\mathbb{A} u)$ is just the scalar product of $z$ and the measure $\mathbb{A} u$, since the weak\star\ limit of  $z:\mathbb{A} u_l$ is $z:\mathbb{A} u$ as $\mathbb{A} u_l\overset{*}{\rightharpoonup} \mathbb{A} u$.

If $u\in W^{\mathbb{A},1}$, then we can take a smooth approximating sequence in the norm topology by Proposition \ref{prop:strict} and obtain by partial integration that $(z,\mathbb{A} u)=z:\mathbb{A} u$.

If $\phi\in C^1(\overline{\Omega})$, then we have $\phi(z,\mathbb{A}u)=(\phi z,\mathbb{A}u)$ by definition.
\end{example}

\begin{proposition}\label{anz pairing}
Let $z$ and $u$ be as above. Then the distribution $(z,\mathbb{A} u)$ is a bounded Radon measure in $\Omega$.

For every Borel measurable $U\subset \Omega$ we have the bound \begin{align}|(z,\mathbb{A} u)|(U)\leq \norm{z}_{L^\infty(\Omega,\R^{m\times n})}|\mathbb{A} u|(U).\end{align}\end{proposition}

\begin{proof}
Let $u_l\in C^\infty(\Omega,\R^m)$ and $u_l\rightarrow u$ in the strict topology in $\BVA(\Omega)\cap L^{p'}(\Omega,\R^m)$ as in Proposition~\ref{prop:strict}. Then we have that \begin{align}
&(z,\mathbb{A} u_l)(\phi)=\int_\Omega \phi z:\mathrm{D}u_l\dx=\int_\Omega \phi Az:\mathrm{D}u_l\dx=\int_\Omega\phi z:\mathbb{A} u_l\dx\leq \norm{\phi}_{\sup}\norm{z}_{L^\infty(\Omega)}\norm{\mathbb{A} u_l}.\end{align} Also $(z,\mathbb{A} u_l)(\phi)$ converges by definition, hence the functional extends to $C(\Omega)$, and by the Riesz representation theorem, it is a bounded Radon measure.

This also shows the bound for $U=\Omega$. For general open $U$ with Lipschitz boundary, it follows by restriction with the exact same argument. In particular, this shows by the Besicovitch derivation theorem \cite[Thm.\ 2.22]{Ambrosio}, that $(z,\mathbb{A}u)$ must be absolutely continuous with respect to $|\mathbb{A}u|$ and its density must have absolute value $\leq \norm{z}_{L^\infty}$ almost everywhere.\end{proof}

We prove some auxiliary convergence statements for later.

\begin{lemma}\begin{itemize}
\item[a)]  If\phantomsection\label{cor:pairing} $z_j\!\overset{\ast}{\rightharpoonup}\! z$ in $L^\infty(\Omega,\R^{m\times n})$ and $\div z_j\!\rightharpoonup \div z$ in $L^p(\Omega,\R^m)$ (if $p=\infty$ weak\star-convergence also suffices), then \begin{align}\int_\Omega (z_j,\mathbb{A} u)\rightarrow\! \int_\Omega (z,\mathbb{A} u)\end{align} for any $u\in \BVA(\Omega)\cap L^{p'}(\Omega,\R^m)$ (with $\frac{1}{p}+\frac{1}{p'}=1$).
\item[b)] For all fixed $z\in X_p(\Omega)$ we have that $\int_\Omega (z,\mathbb{A} u)$ is continuous with respect to strict convergence of $u$ in $\BVA(\Omega)\cap L^{p'}(\Omega,\R^m)$  (for $\frac{1}{p}+\frac{1}{p'}=1$).

 \end{itemize}
\end{lemma}
\begin{proof}
\textbf{a)} Pick some closed $\Omega'\subset\subset \Omega$ such that $|\mathbb{A}u|(\Omega\backslash \Omega')<\eps$ and choose $\phi\in C_c^1(\Omega)$ such that $\phi$ is $1$ on $\Omega'$ and $1\geq \phi\geq 0$. Then we have that \begin{align}
\int_\Omega \phi(z_j,\mathbb{A}u)\rightarrow\! \int_\Omega \phi(z,\mathbb{A}u)\end{align}
by definition. Since we can also estimate \begin{align}
\left|\int_\Omega \phi(z_j,\mathbb{A} u)-\int_\Omega (z_j,\mathbb{A}u)\right|\leq |\mathbb{A}u|(\Omega\backslash\Omega')\times \sup_j \norm{z_j}_{L^\infty}\leq \eps\sup_j \norm{z_j}_{L^\infty}
\end{align}
and the same estimate holds for $z$, we get the statement.\smallskip

\textbf{b)} Take $\phi$ and $\Omega'$ as in a) and let $u_l\rightarrow u$ strictly. We have  \begin{align}
\limsup_{l\rightarrow \infty} |\mathbb{A} u_l|(\Omega\backslash \Omega')\leq |\mathbb{A}u|(\Omega)-|\mathbb{A}u|(\Omega')\leq\eps,
\end{align}
by lower semicontinuity. This gives that \begin{align}
\mel\limsup_{l\rightarrow \infty}\left|\int_\Omega  (z,\mathbb{A} u_l)-\int_\Omega (z,\mathbb{A} u)\right|\\
&\leq \left(\limsup_{l\rightarrow \infty}|\mathbb{A} u_l|(\Omega\backslash \Omega')+\eps\right)\norm{z}_{L^\infty}+\left|\int_\Omega \phi(z,\mathbb{A}(u_l-u))\right|\rightarrow 0,\end{align} where the second summand goes to $0$ by the definition of the pairing.\end{proof}

One can characterise the absolutely continuous part  of this measure:

\begin{definition}\label{theta}
We set \begin{align}\theta(z,\mathbb{A} u):=\frac{\mathrm{d}(z,\mathbb{A}u)}{\mathrm{d}|\mathbb{A}u|}\label{def theta}
\end{align} (it follows from Proposition \ref{anz pairing} that $(z,\mathbb{A} u)<<|\mathbb{A} u|$ and hence the density exists).
\end{definition}

\begin{proposition}\label{density}
We have that \begin{align}\theta(z,\mathbb{A} u)=z:\frac{\mathrm{d}\mathbb{A} u}{\mathrm{d}|\mathbb{A} u|}\text{ a.e.\ w.r.t.\  $|\mathbb{A} u|^a$.}\end{align} 

Also $(z,\mathbb{A} u^s):=(z,\mathbb{A} u)- z:(\mathbb{A}u)^a$ is absolutely continuous with respect to $|\mathbb{A} u|^s$.
\end{proposition}

The proof requires the following Lemma.
\begin{lemma}\label{theta app}
Let $z_l$ be such that \begin{align}
C^1(\Omega,\R^{m\times n})\ni z_l\!\overset{*}{\rightharpoonup}\! z\text{ in $L^\infty(\Omega,\R^{m\times n})$}\end{align} and \begin{align}
\div z_l\!\rightharpoonup\! \div z\text{ in $L^p(\Omega',\R^{m\times n})$ (resp.\ $\div z_l\overset{\ast}{\rightharpoonup} \div z$ for $p=\infty$)}\end{align} for every open $\Omega'\subset\subset \Omega$, then $\theta(z_l,\mathbb{A} u)\overset{*}{\rightharpoonup} \theta(z,\mathbb{A} u)$ in $L^\infty(\Omega,|\mathbb{A} u|)$.
\end{lemma}
\begin{proof}
By definition we have that $\int_\Omega\phi(z_l,\mathbb{A} u)\rightarrow \int_\Omega \phi(z,\mathbb{A} u)$ for $\phi\in C_c^1(\Omega)$, so convergence holds in duality with $C_c^1(\Omega)$. As \begin{align}\norm{\theta(z_l,\mathbb{A} u)}_{L^\infty(\Omega,|\mathbb{A} u|)}\leq \norm{z_l}_{L^\infty}<\infty,\end{align} there is a weakly\star convergent subsequence of the $\theta(z_l,\mathbb{A} u)$ in $L^\infty(\Omega,|\mathbb{A} u|)$. It is easy to see that $C_c^1$-functions are dense in $L^1(\Omega,|\mathbb{A} u|)$ and hence the limit is $\theta(z,\mathbb{A}u)$. \end{proof}

\begin{proof}[Proof of Proposition \ref{density}]
We want to use the Lemma and the fact that for $z\in C^1(\overline{\Omega},\R^{m\times n})$ the statement is trivially true because the measure is just the pointwise product by Example \ref{pairingtrivial}.

We can find $z_l$ as in the lemma by extending $z$ as $0$ to $\R^{n}$ and using mollification. It is easy to see that the weak\star limit in $L^\infty(|\mathbb{A}u|)$ and the weak\star limit with respect to the Lebesgue measure must be the same a.e.\ w.r.t.\ $|\mathbb{A}u|^a$.\smallskip

The second statement in the proposition trivially follows from the first.\end{proof}

In the case $\mathbb{A}=\mathrm{D}$, there are also known characterisations of the density on the singular part, see e.g.\ \cite[Prop.\ 1.6]{Bouchitte} (see also the original unpublished work \cite{anzellotti1983traces}). For other differential operators, such characterisations are an open problem to the best of the author's knowledge.

\subsection{Green's formula}
We can generalize the Green's formula defined for $C^1$ functions in \eqref{Green1} to the space $X_p$, which first requires some notion of trace for $X_p$:

\begin{proposition}\label{normal trace} 
Let $p\in [1,\infty]$. The normal trace on $C^1(\overline{\Omega},\R^{m\times n})$, defined by \begin{align}
(\scalar{z}{\nu})_j:=\sum_{i}z_{ji}\nu_i
\end{align}
on $\de \Omega$, where $\nu_i$ are the components of the (outer) unit normal, has a continuous and linear extension to $X_p(\Omega)$. For every $z\in X_p(\Omega)$ it holds that  \begin{align}\norm{\scalar{z}{\nu}}_{L^\infty(\de\Omega,\R^m)}\leq \norm{z}_{L^\infty(\Omega,\R^{m\times n})}.\end{align} 
Furthermore, for $\phi\in (W^{1,1}\cap L^{p'})(\Omega,\R^m)$ (with $\frac{1}{p}+\frac{1}{p'}=1$), we also have a Green's formula \begin{align}
\int_\Omega \scalar{\div(z)}{\phi}+z:\mathrm{D} \phi\dx=\int_{\de\Omega} \scalar{\phi}{\scalar{z}{\nu}}\dd\mathcal{H}^{n-1}.\label{green regular}\end{align}\end{proposition}
\begin{proof}
This follows by applying \cite[Thm.\ 1.1 and 1.2]{Anzellotti1983PairingsBM} to each row of $z$. The Green's formula is part of the construction in the proof there.
\end{proof}

\begin{proposition}
Let $\mathbb{A}$ be $\mathbb{C}$-elliptic, $z\in X_p(\Omega)$ and $u\in\BVA(\Omega)\cap L^{p'}$ for $\frac{1}{p}+\frac{1}{p'}=1$ and $p'<\infty$. Then we have the Green's formula \begin{align}
\int_{\de\Omega} \scalar{u}{\scalar{z}{\nu}}\dd\mathcal{H}^{n-1}=\int_\Omega(z,\mathbb{A} u)+\int_\Omega \scalar{u}{\div(z)}\dx.\label{Greens form}
\end{align}
\end{proposition}
\begin{proof}
For $u\in W^{1,1}\cap L^{p'}$, this is just the formula from the Proposition above, because the Anzelotti pairing is just the pointwise product by Example \ref{pairingtrivial}.

For general $u$, we strictly approximate in $ \BVA(\Omega)\cap L^{p'}(\Omega,\R^m)$ with $C^\infty(\overline{\Omega},\R^m)$-functions, which is possible by Lemma \ref{strict 2} and makes both sides converge by Thm.\ \ref{trace} and Lemma \ref{cor:pairing} b).\end{proof}

In general, we do not have a Green's formula if $\mathbb{A}$ is not $\mathbb{C}$-elliptic, since there is no trace, but if the normal trace of $z$ is zero, partial integration still works:

\begin{proposition}\label{int by parts}
Suppose that $z\in X_p(\Omega)$ is such that $\scalar{z}{\nu}=0$, then for all $u\in \BVA(\Omega)\cap L^{p'}(\Omega,\R^m)$ with $\frac{1}{p}+\frac{1}{p'}=1$ it holds true that \begin{align}
\int_\Omega (z,\mathbb{A}u)=-\int_\Omega \scalar{u}{\div z}\dx.
\end{align}
\end{proposition}
\begin{proof}
We use Lemma \ref{approx z} below, for such smooth $z_k$, the formula clearly holds, and we can take the limit of the right-hand side. By Lemma \ref{cor:pairing} a), we can also take the limit of the left-hand side.\end{proof}

\begin{lemma}\phantomsection\label{approx z} For every $z\in X_p(\Omega)$ with $\scalar{z}{\nu}=0$ there are $z_k\in  C_c^\infty(\Omega,\R^{m\times n})$ such that \begin{align}&z_k\overset{*}{\rightharpoonup} z\text{ in $L^\infty(\Omega,\R^{m\times n})$}\\
&\div z_k\rightarrow \div z\text{  in $L^{p}(\Omega,\R^{m})$}
\end{align} (resp.\ $\overset{*}{\rightharpoonup}$ if $p=\infty$).
\end{lemma}\begin{proof}

One can partition $\Omega$ into finitely many star-shaped domains \cite[Prop.\ 2.5.4]{carbone2019unbounded}, so by using a partition of unity, we may assume that $z$ is supported in a star-shaped open subdomain $\Omega'\subset \Omega$. 

Next, we note that by the Green's formula in Proposition \ref{normal trace}, the extension $\overline{z}$ of $z$ by $0$ is in $X_p(\R^{n})$ and that the distributional divergence of $\overline{z}$ is the extension of $\div z$ by $0$.

Taking a sequence of standard mollifiers $\eta_l\xrightharpoonup{*} \delta_0$ supported in $B_{\frac{1}{l}}(0)$ and setting \begin{align}\label{def z'}
 z_l':=\overline{z}*\eta_l\end{align} gives that $z_l'$ is smooth, supported in $\Omega'+B_{\frac{1}{l}}(0)$ and lies in $X_p(\R^{n})$.

Now let $s$ be a center point of $\Omega'$ and let $d:=\dist(s,\de\Omega')$, then we define \begin{align}\label{def z}
z_l(x):=z_l'\left(s+\frac{d+\frac{2}{l}}{d}(x-s)\right),\end{align} this function is smooth. It has compact support in $\Omega'$
because $\Omega'$ is star-shaped and $z_l'$ is supported in $\Omega'+B_{\frac{1}{l}}(0)$.

We have that \begin{align}\div z_k=\frac{d+\frac{2}{l}}{d}(\div \overline{z}*\eta_k)\left(s+\frac{d+\frac{2}{l}}{d}(\cdot-s)\right)\end{align} by the chain rule and the fact that the extension of $z$ lies in $X_p(\R^n)$. It is easy to see that this converges to $\div z$ in $L^p(\Omega')$ and that $z_k$ converges to $z$ in every $L^q(\Omega')$. As $z_k$ is uniformly bounded by definition it also converges weakly\star in $L^\infty$.\end{proof}

\section{Main results for Neumann boundary conditions}\label{section 4}

We shall always assume the following about $f$:

\begin{assumption}\label{assf} Let $f:\overline{\Omega} \times \R^{m\times n}\rightarrow \R_{\geq 0}$ be Borel measurable,  convex in the second variable and assume that there are $c_0,C_0>0$ with \begin{align}
c_0|Ay|\leq f(x,y)\leq C_0(1+|Ay|)\end{align}
for all $x,y$ (with the constants not depending on $x,y$). 

\end{assumption}

We remark that assuming positivity of $f$ is not restrictive, because adding a constant to the function corresponds to adding a constant to the functional $F$ (defined below), which does not change the behaviour.

We further remark that the growth condition implies that $f$ depends only on $x$ and $Ay$ and that it is Lipschitz in $y$ with the same constant $C_0$. Assuming that $f$ only depends on $Ay$ is not restrictive, because we only care about functionals of the type $\int_\Omega f(x,\mathbb{A}u)\dx$ and $A^2=A$ as justified in \ref{bas ass}. 

Sometimes we also want $f$ to be a bit more regular: \begin{assumption}\label{assf2}
Suppose $f$ fulfills the assumption \ref{assf}. Further suppose that $f$ is continuous in the joint variable and that the limit in the definition \ref{def rec} of the recession function exists for all $(x,y)\in \overline{\Omega}\times \R^{m\times n}$.
\end{assumption}

If we want to assume \ref{assf2}, we will always explicitly mention it.

\begin{definition}We\label{definition F} define \begin{align}
\mathcal{Z}:=\left\{z\in  X_2(\Omega)\mid f^*(x,z(x))\in L^1(\Omega)\right\},
\end{align} 
where the Legendre-transform is taken with respect to the last variable of $f$ only.

For  $u\in  L^{2}(\Omega,\R^m)$ we define\begin{align}F_{dual}(u):=\begin{cases}
\sup_{z\in\mathcal{Z}}\int_\Omega -f^*(x,z(x))\dx+\int_\Omega (z,\mathbb{A} u)&\text{if $u\in \BVA(\Omega)$}\\
+\infty &\text{else.}\end{cases}\end{align}
\end{definition}

\begin{definition}
For $u\in L^2(\Omega,\R^m)$, we define the functional \begin{align}F_{relaxed}(u):=\inf_{W^{\mathbb{A},1}(\Omega)\ni u_l\rightarrow u \text{ in } L^{2}}\liminf_{l\rightarrow \infty}\int_\Omega f(x,\mathbb{A}u_l(x))\dx.\label{def relax}\end{align}

This is called the relaxation of the functional $\int_\Omega f(x,\mathbb{A}u)\dx$.
\end{definition}

Note that by Mazur's Lemma and convexity, we get the same functional if we only require weak $L^2$-convergence of the $u_l$. We further note that $F_{relaxed}$ is trivially lower-semicontinuous w.r.t.\ $L^2$-convergence, as one can see by using a diagonal sequence.

Our first result is that both functionals always agree, with essentially no assumption (except convexity and measurability) on $f$.

\begin{theorem}\label{T relax 1}
Assume \ref{assf}. Then \begin{align}F_{relaxed}=F_{dual}.\end{align}
Additionally, for every $u\in \BVA(\Omega)\cap L^2(\Omega,\R^m)$, a sequence $u_l$ attaining the infimum in \eqref{def relax} can be taken in $W^{\mathbb{A},1+\frac{1}{l}}(\Omega)\cap L^2(\Omega,\R^m)$.
\end{theorem}

We postpone the proof to Section \ref{proof section}.

In the case in which $f$ is more regular, more standard formulas for the representation of $F$ hold:

\begin{proposition}\label{resh1}
Assume \ref{assf} and \ref{assf2}. Then for all $u\in \BVA(\Omega)\cap L^2(\Omega,\R^m)$ it  holds that \begin{align}
F_{relaxed}(u)=F_{dual}(u)=\int_\Omega f(x,(\mathbb{A}u)^a)\dx+\int_\Omega f^\infty(x,\frac{(\mathbb{A}u)^s}{|\mathbb{A}u|^s}(x))\dd |\mathbb{A}u|^s(x).
\end{align}

In particular, if $\mathbb{A}u\in L^1(\Omega,\R^{m\times n})$, then $F_{relaxed}(u)=F_{dual}(u)=\int_\Omega f(x,\mathbb{A}u(x))\dx$.
\end{proposition}
\begin{proof}
This can be proven exactly as for functionals depending on the full gradient, see e.g.\ \cite[Thm.\ 11.2]{rindler2018calculus} for a proof without the $L^2$ restriction. Requiring the convergence in $L^2$ is not a restriction, since smooth recovery sequences can be taken to converge in $L^2$ by Proposition \ref{prop:strict}.
\end{proof}

We remark that there are very classical counterexamples showing that this is wrong without the Assumption \ref{assf2} even for the full gradient in one dimension, see e.g.\ \cite[p.\ 513]{Bouchitte}.

Our second main theorem is that we can characterise the subdifferential of $F$ as the set of all $z$ for which equality in the definition of $F$ is attained, as explained in the introduction.

\begin{theorem}\label{main thm 1} Assume \ref{assf}. Then for all $u,r\in L^2(\Omega,\R^m)$ the following conditions are equivalent: \begin{itemize}
\item[a)] $r\in \de F_{dual}(u)$.
\item[b)] There is a $z\in \mathcal{Z}$ such that for all $w\in\BVA(\Omega)\cap L^{2}(\Omega,\R^m)$, we have the following inequality: \begin{align}
\int_\Omega \scalar{w-u}{r}\dx\leq \int_\Omega -f^*(x,z)\dx+\int_\Omega (z,\mathbb{A} w)-F_{dual}( u).\label{b 1}\end{align}

\item[c)] There is a $z\in \mathcal{Z}$ such that for all $w\in\BVA(\Omega)\cap L^{2}(\Omega,\R^m)$, we have the following equality: \begin{align}
\int_\Omega \scalar{w-u}{r}\dx= \int_\Omega -f^*(x,z)\dx+\int_\Omega (z,\mathbb{A} w)-F_{dual}( u).\end{align}

\item[d)] There is a $z\in \mathcal{Z}$ such that $r=-\div z$ and $\scalar{z}{\nu}=0$ (the normal trace $\scalar{z}{\nu}$ was defined in Proposition \ref{normal trace}) and \begin{align}
F_{dual}( u)=\int_\Omega-f^*(x,z)\dx+\int_\Omega (z,\mathbb{A} u).\end{align}
\end{itemize}
\end{theorem}

If one assumes some regularity in the $x$-variable, one can also get a more pointwise characterisation of the subdifferential (recall that the density $\theta$ was defined in \eqref{def theta}). \begin{proposition}\label{f cont 1}
Assume \ref{assf} and \ref{assf2}. Then for all $u,r\in L^2(\Omega,\R^m)$ the following conditions are equivalent: \begin{itemize}
\item[a)] $r\in \de F_{dual}(u)$.
\item[b)] There is a $z\in \mathcal{Z}$ such that \begin{align}
&r=-\div z \quad \text{ in $\Omega$}\\
&\scalar{z}{\nu}=0\quad \text{ on $\de\Omega$}\\
&z\in \de_y f(x,(\mathbb{A} u)^a) \quad \text{$\mathcal{L}^n$-a.e.\ in $\Omega$}\\
& \theta(z,\mathbb{A}u)=f^\infty(x,\frac{(\mathbb{A}u)^s}{|\mathbb{A}u|^s}(x)) \quad \text{$|\mathbb{A}u|^s$-a.e.\ in $\Omega$}.
\end{align}
\end{itemize}
\end{proposition}

We postpone the proofs to Section \ref{proof section}.

By Proposition \ref{brezis}, we conclude the existence of the flow:

\begin{corollary}\phantomsection\label{ex flow}
Assume \ref{assf}. For every $u_0\in L^2(\Omega,\R^m)$, there is a unique solution $u\in C([0,\infty),L^2(\Omega,\R^m))\cap W_{loc}^{1,\infty}((0,\infty),L^2(\Omega,\R^m))$ to the problem \begin{align}
&u(0) =u_0,\\
&\de_tu(t)\in\div\left( \de_y f(x,\mathbb{A}u(t))\right)\quad\text{in $\Omega$}\\
&\scalar{\de_yf(x,\mathbb{A}u)}{\nu}=0\quad\text{on $\de\Omega$}
\end{align}
where the time derivative is taken in a weak sense, and the second and third conditions hold in the sense that the equivalent conditions of Thm. \ref{main thm 1} are fulfilled for $r=-\de_t u$. 
\end{corollary}

We remark that the $A^T$ in the divergence in \eqref{main eq} in the introduction disappears here because $A$ is symmetric and $\de_y f(x,\mathbb{A}u)$ already is in the image of $A=A^T$ by the convention \ref{bas ass}.

\section{The functional with Dirichlet boundary conditions}\label{section 6}
In this subsection, we will consider Dirichlet boundary conditions, which come with the additional complications described in the introduction.


\begin{definition}\phantomsection\label{Def F bd}
Assume \ref{assf}. Let $\mathbb{A}$ be $\mathbb{C}$-elliptic and $u_1\in L^1(\de\Omega,\R^m)$. \begin{itemize} \item[a)] For $u\in L^2(\Omega,\R^m)$ we define \begin{align}
&F_{dual}^{u_1}(u):=\sup_{z\in \mathcal{Z}} \int_\Omega -f^*(x,z)-\scalar{u}{\div z}\dx+\int_{\de\Omega} \scalar{\scalar{z}{\nu}}{u_1}\dd\mathcal{H}^{n-1}.
\end{align}
\item[b)] For $u\in L^2(\Omega,\R^m)$ we define the relaxation \begin{align}
F^{u_1}_{relaxed}(u)=\inf_{u_l\rightarrow u \text{ in } L^2,\: u_l\in W_{u_1}^{\mathbb{A},1}(\Omega)}\liminf_{l\rightarrow \infty}\int_\Omega f(x,\mathbb{A}u_l)\dx.\label{def relax2}
\end{align}\end{itemize}\end{definition}

We make a couple of remarks:\begin{itemize}
\item If $u\in \BVA$, then we have by the Green's formula \eqref{Greens form} that \begin{align}
&F_{dual}^{u_1}(u)=\label{}\sup_{z\in \mathcal{Z}}-\int_\Omega f^*(x,z)\dx+\int_\Omega (z,\mathbb{A} u)+\int_{\de\Omega} \scalar{\scalar{z}{\nu}}{u_1-u}\dd\mathcal{H}^{n-1}.\label{pint def}
\end{align}

\item $F_{dual}^{u_1}$ is convex and $L^2$-lower semicontinuous as a supremum of affine and continuous functionals. Arguing as in the proof of Proposition \ref{basic properties F} with compactly supported $z$, one can easily show that $\mathcal{F}_{dual}^{u_1}(u)=+\infty$ whenever $u\notin \BVA$.

\item It follows from the fact that every $L^1(\de \Omega,\R^m)$ is the trace of a $W^{1,1}\cap L^2$ function (see e.g.\ \cite[Thm.\ 1.4]{MR3525400}) that the infimum in b) is taken over a non-empty set.

\item It follows from Mazur's Lemma that one obtains the same functional if one merely requires weak convergence in $L^2$ in the definition of $F_{relaxed}^{u_1}$
\end{itemize}

We have essentially the same theorems in this case:

\begin{theorem}\label{T relax 2}
Assume \ref{assf}, that $u_1\in L^1(\mathbb{C},\R^n)$ and that $\mathbb{A}$ is $\mathbb{C}$-elliptic. Then \begin{align}
F_{dual}^{u_1}(u)=F_{relaxed}^{u_1}(u).
\end{align}
for all $u\in L^2(\Omega,\R^n)\cap \BVA(\Omega)$. 

Additionally, if $u_1$ can be extended to a function in $W_{u_1}^{\mathbb{A},p}(\Omega)\cap L^2(\Omega,\R^m)$ for some $p>1$, then for each $u\in L^2(\Omega,\R^n)\cap \BVA(\Omega)$ there exists a sequence $u_l\in W^{\mathbb{A},\min(p,1+\frac{1}{l})}(\Omega)\cap L^2(\Omega,\R^m)$ attaining the infimum in \eqref{def relax2}.
\end{theorem}
The proof can be found in Section \ref{proof section}.

\begin{proposition}\label{int rep cont}
Assume \ref{assf} and \ref{assf2}, that $u_1\in L^1(\de\Omega,\R^n)$ and that $\mathbb{A}$ is $\mathbb{C}$-elliptic. Then \begin{align}
&F_{dual}^{u_1}(u)=F_{relaxed}^{u_1}(u)\\
&=\int_\Omega f(x,(\mathbb{A}u)^a)\dx+\int_\Omega f^\infty(x,\frac{(\mathbb{A}u)^s}{|\mathbb{A}u|^s}(x))\dd |\mathbb{A}u|^s(x)+\int_{\de\Omega} f^\infty(x,(u_1-u)(x)\otimes \nu)\dd\mathcal{H}^{n-1}(x)
\end{align}
for all $u\in L^2(\Omega,\R^n)\cap \BVA(\Omega)$. In particular, if $u\in W_{u_1}^{\mathbb{A},1}(\Omega)$, then $F_{relaxed}(u)=F_{dual}(u)=\int_\Omega f(x,\mathbb{A}u(x))\dx$.
\end{proposition}
\begin{proof}
This can be proven almost exactly in the full gradient case, see e.g.\ \cite[Section 5]{Diening} for a sketch without the restriction to $L^2$-convergence. While not explicitly stated there, the proof there can easily be adapted to give $L^2$-convergence as well, the only somewhat non-obvious point is that $u_1\in L^1(\de\Omega,\R^m)$ can be extended to an $W^{1,1}\cap L^2$-function, which is true by e.g.\ \cite[Thm.\ 1.4]{MR3525400}. See e.g.\ \cite[Proposition 3.1]{meyer2025attainment} for some more details in the full gradient case. 
\end{proof}

\begin{theorem}Assume \ref{assf}, that $u_1\in L^1(\de\Omega,\R^n)$ and that $\mathbb{A}$ is $\mathbb{C}$-elliptic. Let $u,r\in L^2$, then the following are equivalent:\label{main thm 2}\itemize
\item[a)]$r\in \de F^{u_1}_{dual}(u)$

\item[b)] There is a $z\in \mathcal{Z}$ with $r=-\div z$ such that for all $w\in \BVA(\Omega)\cap L^2(\Omega,\R^m)$ we have the inequality \begin{align}
\int_\Omega \scalar{w-u}{r}\dx\leq&-\int_\Omega f^*(x,z)\dx+\int_\Omega (z,\mathbb{A}w)\label{var ineq}\\
&-\int_{\de\Omega}\scalar{\scalar{z}{\nu}}{w-u_1}\dd\mathcal{H}^{n-1}-F_{dual}^{u_1}(u).\nonumber
\end{align}

\item[c)] There is a $z\in \mathcal{Z}$ with $r=-\div z$ such that for all $w\in \BVA(\Omega)\cap L^2(\Omega,\R^m)$ we have the equality \begin{align}
\int_\Omega \scalar{w-u}{r}\dx=&-\int_\Omega f^*(x,z)\dx+\int_\Omega (z,\mathbb{A}w)\\
&-\int_{\de\Omega}\scalar{\scalar{z}{\nu}}{w-u_1}\dd\mathcal{H}^{n-1}-F_{dual}^{u_1}(u).\nonumber
\end{align}

\item[d)] There is a $z\in \mathcal{Z}$ with $r=-\div z$ and \begin{align}
F_{dual}^{u_1}(u)=\int_\Omega -f^*(x,z)\dx+\int_\Omega(z,\mathbb{A}u)-\int_{\de \Omega}\scalar{\scalar{z}{\nu}}{u-u_1} \dx.
\end{align}

\end{theorem}

\begin{proposition}\label{f cont 2}
Assume \ref{assf} and \ref{assf2}, let $u_1\in L^1(\de\Omega,\R^m)$ and assume that $\mathbb{A}$ is $\mathbb{C}$-elliptic. Then for all $u,r\in L^2(\Omega,\R^m)$ the following conditions are equivalent: \begin{itemize}
\item[a)] $r\in \de F_{dual}^{u_1}(u)$.
\item[b)] There is a $z\in \mathcal{Z}$ such that \begin{align}
r&=-\div z \quad \text{ in $\Omega$}\\
\scalar{z}{\nu}&\in \de_y f^\infty(x,(u_1-u)\otimes \nu)\quad \text{$\mathcal{H}^{n-1}$-a.e. on $\de\Omega$}\\
z&\in \de_y f(x,(\mathbb{A} u)^a(x)) \quad \text{$\mathcal{L}^n$-a.e.\ in $\Omega$}\\
 \theta(z,\mathbb{A}u)&=f^\infty(x,\frac{(\mathbb{A}u)^s}{|\mathbb{A}u|^s}(x)) \quad \text{$|\mathbb{A}u|^s$-a.e.\ in $\Omega$}.
\end{align}
\end{itemize}
\end{proposition}

From Proposition \ref{brezis}, we conclude the existence of the flow:

\begin{corollary}\label{ex flow 2}
For every $u_0\in L^2(\Omega,\R^m), u_1\in L^1(\de\Omega,\R^m)$, there is a unique solution $u\in C([0,\infty),L^2(\Omega,\R^m))\cap W_{loc}^{1,\infty}((0,\infty),L^2(\Omega,\R^m))$ to the problem \begin{align}
&u(0) =u_0,\\
&\de_tu(t)\in\div (\de_y f(x,\mathbb{A}u(t)))\text{ a.e.,}\\
&\scalar{\de_y f^\infty(x,\mathbb{A}u)}{\nu}\in \de_y f^\infty(x,(u_1-u)\otimes \nu)\text{ on $\de\Omega$},
\end{align}
where the time derivative is taken in a weak sense and the second and third conditions hold in the sense that the equivalent conditions of Thm. \ref{main thm 2} are fulfilled for $r=-\de_t u$.
\end{corollary}

\section{Approximation by $q$-Laplace like flows}\label{section 7}
It is a natural question to ask whether the solutions that we constructed can also be obtained as the limit of the flows for $\int_\Omega f^q(x,\mathbb{A}u(x))\dx$ for $q\searrow 1$, which do not require dealing with measures. The answer to this question is yes if the boundary data can be attained by $W^{\mathbb{A},q}\cap L^2$ functions (otherwise it is obviously impossible).

Let us consider $q>1$ and the solutions to the problems \begin{equation}\begin{aligned}
u_q(0)&=u_0\\
-\de_t u&\in \de F_q(u)\label{defq}
\end{aligned}\end{equation}
with \begin{align}
F_q(u)=\begin{cases}
\int_\Omega f^q(x,\mathbb{A}u(x))\dx & \text{ if $u\in L^2(\Omega,\R^m)\cap W^{\mathbb{A},q}(\Omega)$}\\
+\infty & \text{ if $u\in L^2(\Omega,\R^m)\backslash  W^{\mathbb{A},q}(\Omega)$}
\end{cases}
\end{align}
which exists by Proposition \ref{brezis} as one easily checks that $F_q$ is convex and lower semicontinuous. 

Similarly, if $\mathbb{A}$ is $\mathbb{C}$-elliptic and $u_1\in L^1(\de\Omega,\R^m)$ can be extended to a function in $W^{\mathbb{A},q_0}(\Omega)\cap L^2(\Omega,\R^m)$ for some $q_0>1$, we consider for $q\in (1,q_0]$ the solutions to the problems \begin{equation}\begin{aligned}
u_q(0)&=u_0\\
-\de_t u&\in \de F_q^{u_1}(u)\label{defqu1}
\end{aligned}\end{equation}
with \begin{align}
F_q^{u_1}(u)=\begin{cases}
\int_\Omega f^q(x,\mathbb{A}u(x))\dx & \text{ if $u\in L^2(\Omega,\R^m)\cap W_{u_1}^{\mathbb{A},q}(\Omega)$}\\
+\infty & \text{ if $u\in L^2(\Omega,\R^m)\backslash  W_{u_1}^{\mathbb{A},q}(\Omega)$}
\end{cases}
\end{align}
which again exist by Proposition \ref{brezis}.

Arguing for instance as in the proof of Lemma \ref{subdiff reg} below, it is routine to check that if $f$ is sufficiently regular, then the subdifferential $\de F_q$ (resp.\ $\de F_q^{u_1}$) is given by $-\div\mathrm{D}_y (f^q(x,\mathbb{A}u(x)))$, together with a Neumann-type boundary condition $\scalar{\mathrm{D}_y (f^q(x,\mathbb{A}u(x)))}{\nu}=0$ on $\de\Omega$ (resp.\ a Dirichlet boundary condition $u|_{\de\Omega}=u_1$).

\begin{proposition}\phantomsection\label{p approx}
Assume \ref{assf} and  that $u_0\in  L^2(\Omega,\R^m)$. \begin{itemize}
\item[a)] The solutions to \eqref{defq} converge locally uniformly in time in $C([0,\infty),L^2(\Omega,\R^m))$ to the solution of $\de_tu\in -\de F_{dual}$ with initial datum $u_0$ (as constructed in Corollary \ref{ex flow}) for $q\searrow 1$.

\item[b)] Assume additionally that $\mathbb{A}$ is $\mathbb{C}$-elliptic and that $u_1\in L^1(\de\Omega,\R^m)$ extends to some function in $W^{\mathbb{A},q_0}(\Omega)\cap L^2(\Omega,\R^m)$ with $q_0>1$. Then the solutions to \eqref{defqu1} converge locally uniformly in time in $C([0,\infty),L^2(\Omega,\R^m))$ to the solution of $\de_tu\in -\de F^{u_1}_{dual}$ with initial datum $u_0$ (as constructed in Corollary \ref{ex flow 2}) for $q\searrow 1$.
\end{itemize}
\end{proposition}

\begin{proof}
It is enough to show that the functionals $F_q$ (resp.\ $F_q^{u_1}$) Mosco-converge to to $F_{dual}$ resp.\ $F_{dual}^{u_1}$. Indeed, it is known \cite[Thm.\ 3.66 and 3.74]{Attouch} that Mosco-convergence of functionals implies (local in time) strong convergence of the corresponding gradient flows.

Let us first check the lower bound inequality, i.e.\ that for all $u_q\xrightharpoonup{L^2} u$ it holds that $\liminf_{q\searrow 1} F_q(u_q)\geq F_{dual}$ (resp.\ $\liminf_{q\searrow 1} F_q^{u_1}(u_q)\geq F^{u_1}_{dual}$). This is only non-trivial if $u_q\in W^{\mathbb{A},q}$ (resp.\ $W_{u_1}^{\mathbb{A},q}$). We then have that \begin{align*}\liminf_{q\searrow 1} \int_\Omega f^q(x,\mathbb{A}u_q)\dx\geq \liminf_{q\searrow 1}\int_\Omega qf(x,\mathbb{A}u_q)-(q-1)\dx\geq F_{relaxed}(u)=F_{dual}(u),\end{align*}
where we have used Theorem \ref{T relax 1} and that the relaxation does not change when using weak $L^2$-convergence. The case with boundary conditions works with the same argument using Theorem \ref{T relax 2}.

We move on to check the existence of recovery sequences, i.e.\ that for every $u\in L^2(\Omega,\R^m)$, there is a sequence of $u_q\xrightarrow{L^2} u$ such that $F_q(u_q)\rightarrow F_{dual}(u)$ (resp.\ $F_q^{u_1}(u_q)\rightarrow F^{u_1}_{dual}(u)$). This is trivial if $u\notin \BVA$. If on the other hand $u\in \BVA$, then for every $\eps>0$ there is an $u_\eps\in W^{\mathbb{A},p_\eps}(\Omega)$ for some $p_\eps>1$, which fulfills $\norm{u-u_\eps}_{L^2}<\eps$ and \begin{align*}\left|F_{dual}(u)-\int_\Omega f(x,\mathbb{A}u_\eps)\dx\right|<\eps\end{align*} by the last part of Theorem \ref{T relax 1}.

As $|\mathbb{A} u_\eps|^{p_\eps}$ is integrable, it fulfils $f^q(x,\mathbb{A} u_\eps)\rightarrow f(x,\mathbb{A}u_\eps)$ in $L^1$ as $q\searrow 1$ by e.g.\ dominated convergence, yielding a recovery sequence by taking a diagonal sequence.
For the case with Dirichlet boundary conditions, this follows analogously by taking $u_\eps$ with $u_\eps|_{\de\Omega}=u_1$, which is possible by the last part of Theorem \ref{T relax 2}.
\end{proof}

\section{Proof of the main theorems}\label{proof section}

\subsubsection{Overview and general strategy}
In this section, we will prove the Theorems \ref{main thm 1}, \ref{main thm 2}, \ref{T relax 1} and \ref{T relax 2}.

The main difficulty in the proof will be to show that in the Theorems \ref{main thm 1} and \ref{main thm 2} a) implies the variational inequality in b). To do so, we will show that the operator $\widetilde{\de F_{dual}}$ defined through the condition b) is a subset of the subdifferential and that it fulfills the range condition $\ran(I+\widetilde{\de F_{dual}})=L^2$, which implies by Lemma \ref{monotonicity} that it must already be the entire subdifferential. To do so, we consider a regularised version of the functional, for which we can directly compute the Euler-Lagrange equation and show that we obtain elements of $\widetilde{\de F_{dual}}$ as certain limits. The sequence of functions that we obtain this way will also be an optimal sequence for the relaxation.

To regularise the functional, we shall consider $f_{\lambda}^q$ for $q\searrow 1$ (so that we do not have to work with measures) and $\lambda\searrow 0$, where $f_\lambda$ is the Moreau-Yosida approximation (defined below), so that the function becomes differentiable.

The proofs with and without boundary terms are quite similar to each other, so we focus on the case with boundary terms and comment on the changes that need to be made for the other case.

\subsection{Auxillary lemmata}
We shall first collect some easy, yet useful statements:

Let us first of note that whenever it holds $|z(x)|> C_0$ for the constant $C_0$ from \ref{assf}, then it holds that $f^*(x,z(x))=+\infty$, as one can directly see from using the sequence $tz(x)$ in the definition of the Legendre transform, hence for all $z\in \mathcal{Z}$, we have \begin{align}
\norm{z}_{L^\infty(\Omega,\R^{m\times n})}\leq C_0.\label{bound z}
\end{align}

\begin{lemma}\phantomsection\label{basic properties F} Assume \ref{assf}. Then the functionals have the following properties:\begin{itemize}
\item[a)]$F_{dual}$ is convex and lower semicontinuous w.r.t.\ $L^{2}$-convergence.

\item[b)] If $u\in W^{\mathbb{A},1}(\Omega)\cap L^2(\Omega,\R^m)$, then \begin{align}
F_{dual}(u)\leq \int_\Omega f(x,\mathbb{A}u)\dx,\label{W1 ineq}
\end{align}
in particular $F_{dual}(u)\leq F_{relaxed}(u)$ for all $u\in L^2(\Omega,\R^m)$.

\item[c)] If $u_1\in L^1(\de\Omega,\R^m)$ and $\mathbb{A}$ is $\mathbb{C}$-elliptic, then for all $u\in W_{u_1}^{\mathbb{A},1}(\Omega)\cap L^2(\Omega,\R^m)$, it holds that \begin{align}
F_{dual}^{u_1}(u)\leq \int_\Omega f(x,\mathbb{A}u)\dx,
\end{align}
in particular $F_{dual}^{u_1}(u)\leq F_{relaxed}^{u_1}(u)$ for all $u\in L^2(\Omega,\R^m)$.

\item[d)] We have \begin{align}c_0\norm{\mathbb{A}u}\leq F_{dual}( u)\leq C_0(|\Omega|+\norm{\mathbb{A}u})\label{bound F}\end{align} (with the same constants as in the growth condition \ref{assf}).

\item[e)] If $u_1,u_1'\in L^1(\de\Omega,\R^m)$ are two given boundary data and $\mathbb{A}$ is $\mathbb{C}$-elliptic, then for any $u\in L^2(\Omega,\R^m)$ it holds that \begin{align}
\left|F_{dual}^{u_1}(u)-F_{dual}^{u_1'}(u)\right|\lesssim \norm{u_1-u_1'}_{L^1(\de\Omega,\R^m)}
\end{align} 
(with the convention $\infty-\infty=0$).
 \end{itemize}\end{lemma}\begin{proof}
\textbf{a)} As a supremum of affine functionals, this is automatically convex. 

For lower semicontinuity, we take a sequence $u_l\rightarrow u$. The statement is only nontrivial if $u_l\in \BVA$. First, we consider the case when $u\in \BVA$.

 We only need to take the supremum in the definition of $F$ over compactly supported $z$ since for all  $z\in \mathcal{Z}$ and $\phi \in C_c^\infty(\Omega)$ with $\phi_l\nearrow 1$ it holds that \begin{align}
&\int_\Omega -f^*(x,\phi_l z)\dx+\int_\Omega (\phi_l z,\mathbb{A} u)\rightarrow\int_\Omega -f^*(x,z)\dx+\int_\Omega (z,\mathbb{A} u)
\end{align} by dominated convergence. By the definition of the pairing \begin{align}\int_\Omega (\phi_l z,\mathbb{A}u)=-\int_\Omega \scalar{\div \phi_l z}{u}\dx,\end{align} which is $L^2$-continuous in $u$. Hence, the functional is lower semicontinuous w.r.t.\ $L^{2}$-convergence of $u$ as a supremum of continuous functionals.

The other case where $u\notin \BVA$ follows from the bounds in d) and the lower semicontinuity of the total $\mathbb{A}$-variation.

\textbf{b)} This follows trivially from the pointwise inequality \begin{align}
-f^*(x,z(x))+z(x):\mathbb{A}u(x)\leq f(x,\mathbb{A}u(x)),
\end{align}
and the fact that $(z,\mathbb{A}u)$ is just the pointwise product in this case, as established in Example \ref{pairingtrivial}. The inequality between $F_{relaxed}$ and $F_{dual}$ follows by applying \eqref{W1 ineq} to any sequence attaining the infimum in the definition \eqref{def relax} of $F_{relaxed}$ and using the lower semicontinuity established in part a)

\textbf{c)} This follows in exactly the same way as b) by using the formula \eqref{pint def}.

\textbf{d)} We have \begin{align}
&\int_\Omega -f^*(x,z)\dx+\int_\Omega (z,\mathbb{A} u)=\int_\Omega -f^*(x,z)\dx+\int_\Omega z:(\mathbb{A} u)^a\dx+\int_\Omega (z,\mathbb{A} u^s),\label{middle step}\end{align}
where the pairing with the singular part was defined in Proposition \ref{density}. We may estimate the first group of summands against $\int_\Omega f(x,(\mathbb{A}u)^a)\dx$ as in b). The inequality \eqref{bound z} and Proposition \ref{anz pairing} imply that \begin{align}
\int_\Omega (z,\mathbb{A} u^s)\leq C_0\norm{\mathbb{A}u^s}.
\end{align}
This implies that we can estimate \eqref{middle step} as \begin{align}
\leq \int_\Omega f(x,(\mathbb{A}u)^a)\dx+C_0\norm{(\mathbb{A}u)^s}&\leq C_0\left(|\Omega|+\norm{(\mathbb{A}u)^a}\right)+C_0\norm{(\mathbb{A}u)^s}\\
&\leq C_0\left(|\Omega|+\norm{\mathbb{A}u}\right)
\end{align}
where we used the upper bound in \ref{assf} on $f$.

For the lower bound we note that by the bounds on $f$, any $z\in C_c^1(\Omega,\R^{m\times n})$ with $Az=z$ and $\norm{z}_{L^\infty}\leq c_0$ is in $\mathcal{Z}$ and for such $z$ we have $-f^*(x,z(x))\geq 0$ as a direct calculation shows. Hence, making use of Example \ref{pairingtrivial}, we get that \begin{align}
F_{dual}(u)\geq \sup_{z\in C_c^1(\Omega,\R^{m\times n}),Az=z, \norm{z}_{L^\infty}\leq c_0} \int_\Omega z\,\text{d}\mathbb{A} u=c_0\norm{\mathbb{A} u}.
\end{align}

\textbf{e)} Let $\eps>0$ and $z\in \mathcal{Z}$ be such that \begin{align}
F_{dual}^{u_1}(u)\leq \int_\Omega -f^*(x,z)-\scalar{u}{\div z}\dx+\int_{\de\Omega} \scalar{\scalar{z}{\nu}}{u_1}\dd\mathcal{H}^{n-1}+\eps.
\end{align}
Then by \eqref{bound z}, we can estimate  \begin{align}
F_{dual}^{u_1}(u)&\leq \int_\Omega -f^*(x,z)-\scalar{u}{\div z}\dx+\int_{\de\Omega} \scalar{\scalar{z}{\nu}}{u_1'}\dd\mathcal{H}^{n-1}+C_0\norm{u_1-u_1'}_{L^1(\de\Omega,\R^m)}+\eps\\
&\leq F_{dual}^{u_1'}(u)+C_0\norm{u_1-u_1'}_{L^1(\de\Omega,\R^m)}+\eps,
\end{align}
which shows the statement since $\eps>0$ is arbitrary, and the opposite inequality can be shown in the same way. 
\end{proof}

For a convex function $\phi:\R^d\rightarrow \R$ and $\lambda>0$ the Moreau-Yosida approximation is defined as \begin{align}
\phi_{\lambda}(x)=\inf_{y\in \R^d}\phi(y)+\frac{1}{2\lambda}|x-y|^2.
\end{align}
One can show that this is a continuously differentiable and convex function and that it holds that $\phi_\lambda\leq \phi$ everywhere (see e.g.\ \cite[Section 12.4]{Bauschke}).

We always take the Moreau-Yosida regularisation of $f$ in the $y$-variable only. It is easy to see that $f_\lambda(x,y)$ only depends on $x,Ay$ because the same is true for $f$. Furthermore, $f_\lambda$ is non-negative because $f$ is.

\begin{lemma}\label{main lemma} Assume \ref{assf}, then $f$ and $f_\lambda$ have the following properties:\begin{itemize}
\item[a)] $f$ is Lipschitz in the second variable with the same constant $C_0$ as in the growth bound \ref{assf}. The same holds for the Moreau-Yosida regularisation.
\item[b)] There is a constant $C$ (depending on $f$ but not on the other variables) such that for every $\lambda>0$ and $x,y$  we have  \begin{align}f_\lambda(x,y)\geq f(x,y)-C\lambda .\end{align} 
\end{itemize}\end{lemma}\begin{proof}
\textbf{a)} For $f$ this follows directly from the upper growth bound in \ref{assf} and the monotonicity of the difference quotients. As the upper growth bound also holds for $f_\lambda$ the Lipschitzness also holds for $f_\lambda$ with the same constant.\smallskip

\textbf{b)} Let $y,y'$ be given, then by $f$ being Lipschitz in the second variable we may estimate \begin{align}
f(x,y)-f(x,y')\leq C|y-y'|\leq \frac{1}{2\lambda}|y-y'|^2+C\lambda,
\end{align}
where we applied Young's inequality. After rearranging and taking the minimum over all $y'$ this gives the statement.\end{proof}

\subsection{Main Proof}

For the boundary conditions, we certainly can not use every $u_1\in L^1$ for the regularisation approach, as traces of $W^{\mathbb{A},q}$-functions will in general have some fractional regularity. We shall therefore first prove the Theorem \ref{main thm 2} under the following extra assumption, we will later explain in \ref{general u1} how to generalise this to boundary values in $L^1$. For the proof of Theorem \ref{main thm 1}, this assumption is not necessary.

\begin{temporaryassumption} \label{temp ass} We first assume that $u_1$ can be extended to an $W^{\mathbb{A},s}(\Omega)\cap L^2(\Omega,\R^m)$-function for some $s>1$.
\end{temporaryassumption}
 
We will also always assume in all the subsequent Lemmata that the assumptions of the theorems (i.e.\ \ref{assf} and that $\mathbb{A}$ is $\mathbb{C}$-elliptic in the Dirichlet case) hold without specifically mentioning it anymore.

\subsubsection{Lemmata about the regularisation}

Let $q>1$, then for $u\in W_{u_1}^{\mathbb{A},q}(\Omega)\cap L^2$ (resp.\ $u\in W^{\mathbb{A},q}(\Omega)\cap L^2$) in the case without boundary conditions) we define the functional \begin{align}
G_q(u):=\int_\Omega f_{q-1}^q(x,\mathbb{A}u(x))\dx.
\end{align}
Here $f_{q-1}^q$ is always understood as $(f_{q-1})^q$. We extend this functional to $L^2(\Omega,\R^m)$ as $+\infty$.

\begin{lemma}\phantomsection\label{subdiff reg}

\item[a)] The functional $G_q$ is convex and lower semicontinuous on $L^{2}(\Omega,\R^m)$.
\item[b)] Any $r\in \de G_q(u)$ fulfills the inequality \begin{equation}\begin{aligned}\int_\Omega \scalar{w-u}{r}\dx\leq &\int_\Omega -\left(f_{q-1}^q\right)^*\left(x,\mathrm{D}_yf_{q-1}^q(x,\mathbb{A} u)\right)\dx\label{reg inequality}+\int_\Omega \mathrm{D}_yf_{q-1}^q(x,\mathbb{A} u):\mathbb{A}w\dx\\
&-\int_\Omega f_{q-1}^q(x,\mathbb{A} u)\dx\end{aligned}\end{equation}
for all $w\in W_{u_1}^{\mathbb{A},q}(\Omega)\,\cap\, L^{2}(\Omega,\R^m)$ resp.\ for all $w\in W^{\mathbb{A},q}(\Omega)\,\cap\, L^{2}(\Omega,\R^m)$.
\end{lemma}
\begin{proof}
\textbf{a)} Convexity is clear from the definition. Lower semicontinuity follows from standard arguments as long as the limit lies in $W_{u_1}^{\mathbb{A},q}$ (resp.\ $W^{\mathbb{A},q}$). For the case where the limit does not lie in there, we need to show coercivity.

It follows from the growth bound in \ref{assf} and Lemma \ref{main lemma} b) above that there are positive $c,c'$ not depending on $q,x,y$ such that \begin{align}
c|Ay|-c'(q-1)\leq f_{q-1}(x,y),
\end{align}%
this implies that there is a $C>0$ (depending on $f$ but not on $q$) such that for all small enough $q>1$ we have \begin{align}
\frac{1}{C}|Ay|^q-C\leq f_{q-1}^q(x,y).
\end{align}%
Using Hölder and the upper growth bound on $f$ we get that for some $C$, not depending on $q$, that for small enough $q$ it holds that\begin{align}
\frac{1}{C}\norm{\mathbb{A}u}_{L^q}^q-C\leq G_q(u)\leq C(1+\norm{\mathbb{A}u}_{L^q}^q).\label{bound g}
\end{align}
This lower bound implies that whenever we have a sequence $u_l\in W^{\mathbb{A},q}\cap L^2$ which converges to some $u\notin W^{\mathbb{A},q}$, then we must have $G_q(u_l)\rightarrow +\infty$, since if $\norm{\mathbb{A}u_l}_{L^q}$ were uniformly bounded, then by the lower semicontinuity of the $L^q$-norm we would also have $\norm{\mathbb{A}u}_{L^q}<\infty$.

Furthermore, we see by Mazur's Lemma and the norm-continuity of the trace that if we have a sequence of $u_l\in W_{u_1}^{\mathbb{A},q}\cap L^2$, for which $\norm{\mathbb{A}u_l}_{L^q}$ is bounded, converging in $L^2$ to $u\in W^{\mathbb{A},q}$, then $u$ must have $u_1$ as its trace.

This shows the remaining cases of the lower semicontinuity.

\textbf{b)} If $r\in \de G_q$ then, by the definition of the subdifferential, we have \begin{align}
\int_\Omega \scalar{w}{r}\dx\leq \left(\int_\Omega f_{q-1}(x,\mathbb{A} (u+w))^q-f_{q-1}(x,\mathbb{A} u)^q\dx\right)
\end{align}
for all $w\in W_0^{\mathbb{A},q}(\Omega)\cap L^2(\Omega,\R^m)$ (resp.\ $w\in W^{\mathbb{A},q}(\Omega)$). Since the subdifferential is only defined when $G_q(u)<\infty$ we only have to consider $\mathbb{A} u\in L^q(\Omega,\R^{m\times n})$. We have \begin{align}&\frac{1}{h}\left(\int_\Omega f_{q-1}(x,\mathbb{A} (u+hw))^q-f_{q-1}(x,\mathbb{A} u)^q\dx\right)\rightarrow \int_\Omega \mathrm{D}_yf_{q-1}^q(x,\mathbb{A} u):\mathbb{A} w\dx,\end{align} which follows by dominated convergence and because the difference quotients of convex functions are monotone.

 This gives us that \begin{align}
\int_\Omega \scalar{w}{r}\dx\leq \int_\Omega \mathrm{D}_yf_{q-1}^q(x,\mathbb{A} u):\mathbb{A} w\dx
\end{align}
for all $w\in W_0^{\mathbb{A},q}(\Omega)\cap L^2(\Omega,\R^m)$ (resp.\ $W^{\mathbb{A},q}(\Omega)$).

This inequality can be rewritten as \begin{align}\int_\Omega\scalar{w-u}{r}\dx\leq &\int f_{q-1}^q(x,\mathbb{A} u)-\mathrm{D}_yf_{q-1}^q(x,\mathbb{A} u):\mathbb{A} u\dx\notag\\
&+\int_\Omega \mathrm{D}_yf_{q-1}^q(x,\mathbb{A} u):\mathbb{A} w\dx\notag\\
&-\int_\Omega f_{q-1}^q(x,\mathbb{A} u)\dx,\end{align}
for all $w\in W_{u_1}^{\mathbb{A},q}(\Omega)\cap L^{2}(\Omega,\R^m)$ (resp.\ $W^{\mathbb{A},q}$). The first integrand equals \begin{align}f_{q-1}^q(x,\mathbb{A} u)-\mathrm{D}_yf_{q-1}^q(x,\mathbb{A} u):\mathbb{A} u=-\left(f_{q-1}^q\right)^*\left(x,\mathrm{D}_yf_{q-1}^q(x,\mathbb{A} u)\right)\end{align}
and we conclude.
\end{proof}

\subsection{Proof of b)$\implies$a)}
We now move on to the proof of the Theorems \ref{main thm 1} and \ref{main thm 2}.

We first show that b)$\implies$ a)  in the Theorems \ref{main thm 1} and \ref{main thm 2}. By the Definition \ref{Def F bd} of $F_{dual}^{u_1}$ we have for all $w\in \BVA\cap L^2$ that \begin{align}
-\int_\Omega f^*(x,z)\dx+\int_\Omega (z,\mathbb{A}w)-\int_{\de\Omega}\scalar{\scalar{z}{\nu}}{w-u_1}\dd\mathcal{H}^{n-1}\leq F_{dual}^{u_1}(w),
\end{align}
plugging this into the variational inequality \eqref{var ineq} in b) shows that \begin{align}
\int_\Omega \scalar{w-u}{r}\dx\leq F_{dual}^{u_1}(w)-F_{dual}^{u_1}(u),
\end{align}
the same argument also works for $F_{dual}$.\smallskip

\subsection{Construction of an approximating sequence  for a)$\implies$b)}
Next, we show a)$\implies$ b) in the Theorems \ref{main thm 1} and \ref{main thm 2}, which is the main step. We denote the operator of all  $(u,r)$ fulfilling the condition b) from the theorems by $\widetilde{\de F_{dual}^{u_1}}$ resp.\ $\widetilde{\de F_{dual}}$. By b)$\implies$a) we  have $\widetilde{\de F_{dual}^{u_1}}\subset \de F_{dual}^{u_1}$ resp.\ $\widetilde{\de F_{dual}}\subset \de F_{dual}$. 

To show the reverse inclusion, we show that these operators are already maximal in the sense that for every $v\in L^2(\Omega,\R^m)$, we can find $u,r\in L^2(\Omega,\R^m)$ such that $u+r=v$ and $r\in \widetilde{\de F_{dual}^{u_1}}$ (resp.\ $\widetilde{\de F_{dual}}$). This implies the reverse inclusion by Lemma \ref{monotonicity}.

We pick some $v\in L^2(\Omega,\R^m)$. By the Lemmata \ref{monotonicity} and \ref{subdiff reg} a), there are $u_q$ and $r_q\in \de G_q(u_q)$ such that $u_q+ r_q=v$ for each (sufficiently small) $q>1$. We would like to recover suitable $r,u$ in the limit $q\searrow 1$.
\begin{lemma}\phantomsection\label{bdness u}
The sequence $\norm{u_q}_{L^2(\Omega,\R^m)}+\norm{\mathbb{A}u_q}_{L^q(\Omega,\R^{m\times n})}$ is uniformly bounded for $q\searrow 1$.
\end{lemma}
\begin{proof}

 Let $\overline{u}_1\in W^{\mathbb{A},s}\cap L^2$ denote the extension of $u_1$ to $\Omega$ which exists by the temporary assumption \ref{temp ass}. For $q<s$, we have that $\norm{\mathbb{A}\overline{u}_1}_{L^q}$ (and by \eqref{bound g} also $G_q(\overline{u}_1)$) is bounded uniformly in $q$.
 
 By Lemma \ref{monotonicity}, there is some $\tilde{u}_q$ with $\overline{u}_1-\tilde{u}_q\in \de G_q(\tilde{u}_q)$. By using the definition of the subdifferential, we see that \begin{align}
\norm{\tilde{u}_q-\overline{u}_1}_{L^2}^2\leq G_q(\overline{u}_1)-G_q(\tilde{u}_q)\leq G_q(\overline{u}_1), 
\end{align} 
where we used the positivity of $f_\lambda$. Hence $\norm{\tilde{u}_q}_{L^2}$ is uniformly bounded. By the non-expansiveness statement in Lemma \ref{monotonicity}, we also see that $u_q$ is uniformly bounded in $L^2$. 

By testing the definition of the subdifferential at $u_q$ with $\overline{u}_1$ we see that \begin{align}
\scalar{\overline{u}_1-u_q}{v-u_q}\leq G_q(\overline{u}_1)-G_q(u_q),
\end{align}
hence $G_q(u_q)$ and by \eqref{bound g} also $\norm{\mathbb{A}u_q}_{L^q}$ are uniformly bounded.

In the case without boundary conditions, one can make the same argument with $0$ replacing $\overline{u}_1$.\end{proof}

 Hence, at least along a (non-relabeled) subsequence, these $u_q$'s converge weakly to some $u\in L^2$. Similarly, the $r_q=v-u_q$ are bounded as well and weakly converge along a subsequence to some $r$.

 We would like to take the limit $q\searrow 1$ in the inequality \eqref{reg inequality}, to obtain the desired variational inequality for $(u,r)$ in the limit. For this, we first need to show suitable convergence statements for all the terms.

We set \begin{align}
z_q:=\text{D}_yf_{q-1}^q(x,\mathbb{A}u_q).
\end{align}
Because $f_\lambda(x,y)=f_\lambda(x,Ay)$, we have $Az_q=z_q$. We have \begin{align}
&\int_\Omega |z_q|^2\dx=\int_\Omega\left|qf_{q-1}^{q-1}(x,\mathbb{A} u_q)\text{D}_yf_{q-1}(x,\mathbb{A} u_q)\right|^2\dx\\
\lesssim& \int_\Omega\left(1+|\mathbb{A}u_q|\right)^{2q-2}\dx\\
\lesssim& 1+\norm{\mathbb{A}u_q}_{L^q}^q
\end{align}
by the Hölder inequality, the growth bound on $f$, the fact that $f_\lambda\leq f$ and Lemma \ref{main lemma} a). Hence, along a subsequence, we have $z_q\xrightharpoonup{L^2} z$ for some $z\in L^2$ with $Az=z$.

 Next, we set \begin{align}
z_q':=-\left(f_{q-1}^q\right)^*\left(x,\mathrm{D}_yf_{q-1}^q(x,\mathbb{A} u_q)\right)=f_{q-1}(x,\mathbb{A} u_q)^q-\text{D}_yf_{q-1}^q(x,\mathbb{A} u_q):\mathbb{A} u_q
\end{align}
By using the same bounds for the derivatives as above, we have \begin{align}
&\int_\Omega \left|z_q'\right|\dx
=\int_\Omega \Bigg|f_{q-1}^q(x,\mathbb{A} u_q)-\text{D}_yf_{q-1}^q(x,\mathbb{A} u_q):\mathbb{A} u_q\bigg|\dx\notag\\
&\lesssim \int_\Omega \left(1+|\mathbb{A}u_q|\right)^{q}+|\mathbb{A}u_q|\left(1+|\mathbb{A}u_q|\right)^{q-1}\dx,\phantomsection\label{est z pprime1}\\
&\lesssim 1+\norm{\mathbb{A}u_q}_{L^q}^q
\end{align}
where we made use of the growth bound of $f$ in \ref{assf}.
Hence we conclude that $z_q'\stackrel{\ast}{\rightharpoonup} z'$ in $\mathcal{M}(\Omega)$ along a subsequence.\medskip

Next, we need to show that this $z$ is in $\mathcal{Z}$.

\begin{lemma}\label{est z pprime}
We have \begin{align}
\limsup_{q\searrow 1}\int_\Omega z_q'\dx\leq\int_{\Omega}z'\dx\leq \int_\Omega-f^*(x,z)\dx.
\end{align}
\end{lemma}
\begin{proof}
The first inequality follows from Fatou and the fact that $f_{q-1}(0)^q\leq f(0)^q\leq   C_0^q$ (by \ref{assf}) and that we hence have $z_q'\leq C_0^q$, which is uniformly bounded from above.

For the second inequality, we first note that $f^*(x,z)\geq -f(x,0)\geq -C_0$ by \ref{assf} and that the integral on the right is hence well-defined.

For all  $a\in \Q^{m\times n}$ we have by Fenchel duality \begin{align}
f_{q-1}(x,a)^q\geq a:z_q+z_q'.
\end{align}
 Now $f_{q-1}(x,a)^q$ converges uniformly to $f(x,a)$ by the growth bound \ref{assf} on $f$ and Lemma \ref{main lemma} b). Because weak\star\ convergence preserves the $\geq$ relation, this shows that \begin{align}
a:z+z'-f(x,a)\leq 0
\end{align}
in the sense of measures and hence the density of the absolutely continuous part of this w.r.t.\ the Lebesgue measure is $\leq 0$ a.e.\ and the singular part is purely negative. But since $ \Q^{m\times n}$ is a dense subset of $ \R^{m\times n}$ this is only possible if $(z')^a\leq-f^*(\cdot,z)$ a.e.
\end{proof}

Since $f^*(x,z)\geq -C_0$ pointwise and we conclude from the Lemma that \begin{align}
f^*(x,z)\in L^1(\Omega,\R^{m\times n}).
\end{align}
This also shows that $z\in L^\infty(\Omega,\R^{m\times n})$, since $f^*(x,z(x))<\infty$ implies $|z(x)|\leq C_0$ as argued in \eqref{bound z}.

\subsection{Passage to the limit in $a)\implies b)$}

We are now ready to take the limit $q\searrow 1$ in \ref{reg inequality}. By lower semicontinuity, for all $w\in W_{u_1}^{\mathbb{A},s}$ (resp. all $w\in W^{\mathbb{A},s}$) it holds that \begin{align}
\int_\Omega \scalar{w-u}{r}\dx\leq \liminf_{q\searrow 1} \int_\Omega \scalar{w-u_q}{v-u_q}\dx.
\end{align}
Now for small enough $q$, we can use the inequality \ref{reg inequality} and obtain \begin{align}
\int_\Omega \scalar{w-u}{r}\leq \liminf_{q\searrow 1} &\int_\Omega -\left(f_{q-1}^q\right)^*\left(x,\mathrm{D}_yf_{q-1}^q(x,\mathbb{A} u_q)\right)\dx\nonumber+\int_\Omega \mathrm{D}_yf_{q-1}^q(x,\mathbb{A} u_q):\mathbb{A}w\dx\nonumber\\
&-\int_\Omega f_{q-1}(x,\mathbb{A} u_q)^q\dx.
\end{align}
Inserting the definitions of $z_q$ and $z_q'$ and making use of their convergence and Lemma \ref{est z pprime}, we can further estimate this as \begin{align}
\leq \liminf_{q\searrow 1}\int_\Omega -f^*(x,z)+z:\mathbb{A}w\dx-\int_\Omega f_{q-1}^q(x,\mathbb{A} u_q)\dx.\label{middle step limit}
\end{align}
Now we can make use of Young's inequality and the estimate in Lemma \ref{main lemma} b) to obtain that \begin{align}
\mkern17mu\int_{\Omega}f_{q-1}(x,\mathbb{A}u_q)^q\geq \int_\Omega q f_{q-1}(x,\mathbb{A}u_q)\dx-\left(q-1\right)\mathcal{L}^n(\Omega)\label{ineq f}\geq\int_\Omega f(x,\mathbb{A}u_q)\dx-Cq(q-1).
\end{align}
The second summand on the right-hand side goes to zero and hence we obtain from \eqref{middle step limit} and \eqref{ineq f} that \begin{align}
\int_\Omega \scalar{w-u}{r}\leq \int_\Omega -f^*(x,z)+z:\mathbb{A}w\dx-\limsup_{q\searrow 1}\int_\Omega f(x,\mathbb{A} u_q)\dx.\label{ineq w reg}
\end{align}
We have by the definition of the relaxation that \begin{align}\limsup_{q\searrow 1}\int_\Omega f(x,\mathbb{A} u_q)\dx\geq F^{u_1}_{relaxed}(u)\geq F_{dual}^{u_1}(u),\end{align}
where the second inequality follows from Lemma \ref{basic properties F} c).

This shows that the variational inequality \eqref{var ineq}, (resp.\ \eqref{b 1}) hold for all $w\in W_{u_1}^{\mathbb{A},q}\cap L^2$ (resp.\ $W^{\mathbb{A},q}$) for small enough $q>1$. To proceed further, we distinguish the Dirichlet case and the case without boundary conditions. In the first case, we test this inequality with $w=w_1+tw_0$ for some $w_1\in W_{u_1}^{\mathbb{A},q}$ and $w_0\in W_{0}^{\mathbb{A},q}$ and let $t\rightarrow \pm \infty$, which implies \begin{align}
\int_\Omega\scalar{w_0}{r}=\int_\Omega z:\mathbb{A}w_0\dx.\label{limit eq}
\end{align}
This implies that $r=-\div z$ by taking $w_0\in C_c^\infty$. We use the Green's formula \eqref{Greens form}, to obtain that for $w'\in \BVA\cap L^2$ it holds that \begin{align}
\int_\Omega \scalar{w'}{r}\dx=\int_\Omega (z,\mathbb{A}w')-\int_{\de\Omega}\scalar{w'}{\scalar{z}{\nu}}\dd\mathcal{H}^{n-1}.\label{eq tilde w2}
\end{align}
Adding this to \eqref{ineq w reg}, show that a)$\implies$b) in Theorem \ref{main thm 2}.

In the case without Dirichlet boundary conditions, we can again test \eqref{ineq w reg} with $tw$ for any $w\in W^{\mathbb{A},q}\cap L^2$ to conclude that $\int_\Omega\scalar{w}{r}=\int_\Omega z:\mathbb{A}w\dx$ and that $r=-\div z$
and that by \eqref{green regular} we must have $\scalar{z}{\nu}=0.$

 Now we can use the Green's formula in Proposition \ref{int by parts} together with this and obtain that for all $w'\in \BVA\cap L^2$ we have \begin{align}
\int_{\Omega} \scalar{w'}{r}\dx=\int_\Omega (z,\mathbb{A}w').\label{eq tilde w}
\end{align}
Taking $w'=\tilde{w}-w$ for $w\in W^{1,2}$ and $\tilde{w}\in \BVA\cap L^2$, we obtain that \eqref{b 1} holds for all $\tilde{w}\in\BVA\cap L^2$ by adding \eqref{eq tilde w} to \eqref{ineq w reg}.

This finishes the proof that a)$\implies$b) for Theorem \ref{main thm 1}

\subsection{Proof of the remaining implications}

d)$\iff$c)$\iff$b): Clearly, c) implies b). We have already shown that b) implies $r=-\div z$ and in the case with boundary terms also $\scalar{z}{\nu}=0$.

We test b) with $w=u$ and  obtain \begin{align}
F_{dual}(u)\leq \int_\Omega-f^*(x,z)\dx+\int_\Omega(z,\mathbb{A}u)\end{align}resp.
\begin{align}F_{dual}^{u_1}(u)\leq \int_\Omega-f^*(x,z)\dx+\int_\Omega(z,\mathbb{A}u)-\int_{\de\Omega}\scalar{\scalar{z}{\nu}}{u-u_1}\dd\mathcal{H}^{n-1}.
\end{align}
By the definition of $F_{dual}/F_{dual}^{u_1}$ (and \eqref{pint def} for the case with boundary terms), this implies equality, which is exactly the remaining statement of d). By subtracting this equality from \eqref{eq tilde w} resp.\ \eqref{eq tilde w2}, we obtain the desired equality for c).

\subsection{General boundary data $u_1$}\label{general u1}

For Theorem \ref{main thm 2}, it remains to show a)$\implies$ b) for general $u_1\in L^1(\de \Omega,\R^m)$ which do not fulfil the additional assumption \ref{temp ass}.
We certainly still have b)$\implies$a) here, so we again only need to show that the range condition is fulfilled. For this we take $u_{1,l}\in H^\frac{1}{2}(\de\Omega,\R^m)$ converging to $u_1$ in $L^1$. Since these can be extended to functions in $H^1(\Omega,\R^m)$, the assumption \ref{temp ass} and hence Theorem \ref{main thm 2} certainly hold for these.

We again take some $v$ and $u_l$ such that $v-u_l\in\de F_{dual}^{u_{1,l}}(u_l)$, which exist by Lemma \ref{monotonicity}. By making the same argument as for Lemma \ref{bdness u}, we see again that $\norm{u_l}_{L^2}$ is uniformly bounded.

We denote the corresponding $z$, which the theorem yields, by $z_l$.

Now by testing the condition b) (see \eqref{var ineq}) with $w=0$, one obtains that \begin{align}
\sup_{l}\int_\Omega f^*(x,z_l)\leq \sup_l\int_{\Omega}\scalar{u_l}{v-u_l}\dx-\int_{\de\Omega}\scalar{\scalar{z_l}{\nu}}{-u_{1,l}}\dd\mathcal{H}^{n-1}-F_{dual}^{u_{1,l}}(u_l).\label{bound transform}
\end{align}
By \eqref{bound z}, we see that the second summand is uniformly bounded, the first one is trivially bounded. By e.g.\ taking $z=0$ in the definition of $F_{dual}^{u_1}$ and using  $f^*(x,0)=\sup_y -f(x,y)$ and the growth bound in assumption \eqref{assf}, we see that \begin{align}
F_{dual}^{u_{1,l}}(u_l)\geq -\int_\Omega f^*(x,0)\dx= \int_\Omega \inf_y f(x,y)\dx\geq 0. 
\end{align}
Since by the growth bound \ref{assf} of $f$ we also have $f^*(x,z(x))\geq -C_0$, we hence see that $\norm{f^*(x,z_l)}_{L^1}$ is uniformly bounded.

As $|z_l|\leq C_0$ a.e.\ by \ref{bound z}  there is a weak\star limit $z$ with $Az=z$. By the same argument as in the proof of Lemma \ref{est z pprime}, we must have  \begin{align}
\limsup_l\int_\Omega -f^*(x,z_l)\dx\leq \int_\Omega -f^*(x,z)\dx.
\end{align}
Now we can take the limit in the variational inequality \ref{var ineq} for $u_{1,l}$ and obtain that \begin{align}
\int_{\Omega}\scalar{w-u}{v-u}\dx\leq& \int_\Omega-f^*(x,z)\dx+\int_\Omega (z,\mathbb{A}w)-\int_{\de\Omega} \scalar{\scalar{z}{\nu}}{w-u_1}\dd\mathcal{H}^{n-1}\\
&-\liminf_{l\rightarrow \infty} F_{dual}^{u_{1,l}}(u_l).\nonumber
\end{align}
We can make use of Lemma \ref{basic properties F} e) to estimate this as \begin{align}
\int_{\Omega}\scalar{w-u}{v-u}\dx\leq&\int_\Omega-f^*(x,z)\dx+\int_\Omega (z,\mathbb{A}w)-\int_{\de\Omega} \scalar{\scalar{z}{\nu}}{w-u_1}\dd\mathcal{H}^{n-1}\\
&-\liminf_{l\rightarrow \infty} F_{dual}^{u_{1}}(u_l)+C\norm{u_{1,l}-u_1}_{L^1}.\nonumber
\end{align}
By the lower semicontinuity of $F_{dual}^{u_1}$, we conclude that the inequality \ref{var ineq} in b) must hold. From here on, one can proceed exactly as above. This finishes the proof of Thm.\ \ref{main thm 2}.

\subsection{Proof of Thm.\ \ref{T relax 1} and \ref{T relax 2}}

We only show Thm.\ \ref{T relax 2}, the other one works with the same argument and is in fact even slightly easier. We shall also first assume \ref{temp ass} again.
 We reuse the sequences $u_l$ from the proof that a)$\implies$b), in the proof there we have actually shown that for all $w\in \BV^\mathbb{A}\cap L^2$ it holds that \begin{equation}\begin{aligned}
\int_\Omega\scalar{w-u}{r}\leq& \int_\Omega -f^*(x,z)\dx+\int_{\Omega}(z,\mathbb{A}w)\label{relax ineq}-\int_{\de\Omega}\scalar{\scalar{z}{\nu}}{w-u_1}\dd\mathcal{H}^{n-1}\\
&-\limsup_{q\searrow 1} \int_\Omega f(x,\mathbb{A}u_q)\dx,
\end{aligned}\end{equation}
 indeed one obtains this from adding \eqref{ineq w reg} to \eqref{eq tilde w2}. 

We use \eqref{relax ineq} and test it with $w=u$ to obtain that \begin{align}
\limsup_{q\searrow 1}\int_{\Omega}f(x,\mathbb{A}u_q)\dx&\leq \int_{\Omega} -f^*(x,z)\dx+\int _\Omega (z,\mathbb{A}u)-\int_{\de\Omega}\scalar{\scalar{z}{\nu}}{u-u_1}\dd\mathcal{H}^{n-1}\\
&\leq F_{dual}^{u_1}(u).
\end{align}
By Lemma \ref{basic properties F} c),  it follows that \begin{align}
\limsup_{q\searrow 1}\int_{\Omega}f(x,\mathbb{A}u_q)\dx=F^{u_1}_{relaxed}(u)=F_{dual}^{u_1}(u).
\end{align}
This shows the relaxation statement for all $u$ for which there is a $v$ of the form $v=u+r$ with $r\in \de F_{dual}^{u_1}(u)$.  By Lemma \ref{monotonicity}, these are all $u\in \dom(\de F_{dual}^{u_1})$. Now, by using Lemma \ref{subdiff approx} and picking a diagonal sequence, we conclude. 

A general $u_1\in L^1(\de \Omega,\R^m)$ can be extended to a function $\overline{u}_1\in W^{1,1}\cap L^2$ (cf.\ \cite{MR3525400}). We can find $u_{1,l}\in W^{1,2}$ converging to $\overline{u}_1$ in $W^{1,1}\cap L^2$. Then \begin{align}
\norm{u_{1,l}-u_1}_{L^1(\de \Omega,\R^m)}\rightarrow 0.
\end{align}
As we have already proven the theorem for such regular boundary values, we can find $u_{k,l}\in W_{u_{1,l}}^{\mathbb{A},1}\cap L^2$ such that \begin{align}
\lim_{k\rightarrow \infty} \int _\Omega f(x,\mathbb{A}u_{k,l})\dx\rightarrow F_{dual}^{u_{1,l}}(u).\label{w12 approx}
\end{align}
We want to use a diagonal sequence among $u_{k,l}+\overline{u}_1-u_{1,l}$ (which have trace $u_1$) as an optimal sequence. By the Lipschitzness of $f$ we can estimate \begin{align}
&\int_\Omega \big|f(x,\mathbb{A}u_{k,l})-f(x,\mathbb{A}(u_{k,l}+\overline{u}_1-u_{1,l}))\big|\dx\lesssim \norm{\mathbb{A}(\overline{u}_1-u_{1,l})}_{L^1}\rightarrow 0.\label{sim approx}
\end{align}
Together with Lemma \ref{basic properties F} e), we conclude from \eqref{sim approx} and \eqref{w12 approx} by taking a diagonal sequence.\newline\vphantom{a} \hfill\qedsymbol

\subsection{Proof of Proposition \ref{f cont 1} and \ref{f cont 2}}
We shall require the following Lemma:

\begin{lemma}
Assume \ref{assf2} holds. Then for $z\in \mathcal{Z}$ and $u\in \BVA(\Omega)\cap L^2(\Omega,\R^m)$ we have  \begin{align}
\theta(z,\mathbb{A}u)\leq f^\infty(x,\frac{(\mathbb{A}u)^s}{|\mathbb{A}u|^s}(x)) \quad \text{ for $|\mathbb{A}u|^s$-a.e.\ $x\in \Omega$}\label{theta ineq}
\end{align} 
and
\begin{align}
\scalar{w(x)}{\scalar{z}{\nu}}\leq f^\infty(x,w(x)\otimes \nu)\quad \text{ for $\mathcal{H}^{n-1}$-a.e.\ $x\in \de\Omega$},\label{ntrace ineq}
\end{align}
for all $w(x)\in L^1(\de\Omega,\R^m)$.
\end{lemma}
\begin{proof}
We first show \eqref{theta ineq}, we extend $z$ by $0$ to $\R^n$ and convolute it with compactly supported smooth standard mollifiers $\eta_l$ and set $z_l=z*\eta_l$. The Lemma \ref{theta app} is applicable to this approximating sequence, it hence suffices to show that $\limsup_{l\rightarrow \infty} \theta(z_l,\mathbb{A}u)(x)\leq f^\infty(x,\frac{(\mathbb{A}u)^s}{|\mathbb{A}u|^s}(x))$ for $|\mathbb{A}u|^s$-a.e.\ $x\in \Omega$, since weak\star-convergence preserves the $\leq$ relation.

We also have $\theta(z_l,\mathbb{A}u)(x)=z_l:\frac{(\mathbb{A}u)^s}{|\mathbb{A}u|^s}$ as the pairing is just the pointwise product for regular $z_l$ by Example \ref{pairingtrivial}.

Since $f^*(x,z(x))<\infty$ almost everywhere, we also have, using Fenchel duality, whenever $\diam \supp\eta_l<\dist(x,\de\Omega)$ that \begin{align}
z_l(x):\frac{(\mathbb{A}u)^s}{|\mathbb{A}u|^s}(x)=\int_{\R^n}\eta_l(x-y) z(y):\frac{(\mathbb{A}u)^s}{|\mathbb{A}u|^s}(x)\,\dy\leq \int_{\R^n} \eta_l(x-y) f^\infty(y,\frac{(\mathbb{A}u)^s}{|\mathbb{A}u|^s}(x))\,\dy.
\end{align}
By the assumed continuity of $f^\infty$, the right-hand side converges to $f^\infty(x,\frac{(\mathbb{A}u)^s}{|\mathbb{A}u|^s}(x))$ when $l\rightarrow +\infty$, which shows \eqref{theta ineq}. 

We move on to \eqref{ntrace ineq}, which we first show for the case in which $w=a$ is a constant. 
This special case can be concluded from the fact that $\scalar{a}{\scalar{z}{\nu}}$ is the normal trace of the vector field $az=\sum_{i=1}^m a_iz_{i\bullet}$. Indeed this vector field has the property that for all $b\in \R^n$ it holds that $\scalar{az(y)}{b}\leq f^\infty(y,a\otimes b)$ pointwise a.e.\ in $\Omega$ by \eqref{Fenchel} and because $\scalar{az}{b}=z:a\otimes b$. Together with suitable limit formulas (see e.g.\ the unpublished preprint \cite[Prop 2.2]{anzellotti1983traces} (or alternatively the published paper \cite{comi2024representation})) for the normal trace and the assumed continuity of $f^\infty$, we hence see that it holds that \begin{align}
\scalar{a}{\scalar{z}{\nu}}=\scalar{az}{\nu}\leq f^\infty(x, a\otimes \nu).
\end{align}
Now to recover general $w$, we note that \eqref{ntrace ineq} must also hold for simple functions because it holds for constants. Since \eqref{ntrace ineq} is stable under pointwise convergence and every $w\in L^1$ can be obtained as a pointwise a.e.\ limit of simple functions, we conclude.
\end{proof}

\begin{proof}[Proof of Proposition \ref{f cont 1} and \ref{f cont 2}]
We only prove Proposition \ref{f cont 2}, the other case uses the same argument but is slightly easier.

\textbf{a)$\implies$b)}
By Fenchel duality, we have  \begin{align}
z(x):(\mathbb{A}u)^a(x)-f^*(x,z(x))\leq f(x,(\mathbb{A}u)^a(x))
\end{align}
where equality holds precisely if $z(x)\in \de_yf(x,(\mathbb{A}u)^a(x))$.
Combining with Proposition \ref{density} and \eqref{theta ineq}, we see that \begin{align}
\int_\Omega (z,\mathbb{A}u)-\int_\Omega f^*(x,z(x))\leq \int_\Omega f(x,(\mathbb{A}u)^a)\dx+\int_\Omega f^\infty(x,\frac{(\mathbb{A}u)^s}{|\mathbb{A}u|^s}(x))\dd |\mathbb{A}u|^s(x),
\end{align}
and equality implies that $z(x)\in \de_yf(x,(\mathbb{A}u)^a(x))$ and $\theta(z,\mathbb{A}u)=f^\infty(x,\frac{(\mathbb{A}u)^s}{|\mathbb{A}u|^s}(x))$ hold a.e.

Combining this with \eqref{ntrace ineq}, Proposition \ref{int rep cont} and the equivalent condition d) in Theorem \ref{main thm 2}, we conclude that this equality must indeed hold and that we must furthermore have \begin{align}
\int_{\de\Omega} \scalar{\scalar{z}{\nu}}{u_1-u}\dd\mathcal{H}^{n-1}=\int_{\de\Omega} f^\infty(x,(u_1-u)(x)\otimes \nu)\dd\mathcal{H}^{n-1},
\end{align}
which, after using \eqref{ntrace ineq} and \eqref{Fenchel} again implies the statement. 

\textbf{b)$\implies$a)}
It suffices to show that b) implies the condition d) from Theorem \ref{main thm 1} resp.\ \ref{main thm 2}. For this, we can first use that by Fenchel duality we have \begin{align}
z(x):(\mathbb{A}u)^a(x)-f^*(x,z(x))=f(x,(\mathbb{A}u)^a(x))
\end{align}
and hence, using Proposition \ref{density} we see that \begin{align}
\int_\Omega (z,\mathbb{A}u)-\int_\Omega f^*(x,z)\dx=\int_\Omega f(x,(\mathbb{A}u)^a)\dx+\int_\Omega f^\infty(x,\frac{(\mathbb{A}u)^s}{|\mathbb{A}u|^s}(x))\dd |\mathbb{A}u|^s(x).
\end{align}
Similarly, we see from \eqref{Fenchel} that $\scalar{\scalar{z}{\nu}}{u-u_1}=f^\infty(x,(u-u_1)\otimes \nu)$ holds on $\de\Omega$.

Combining with Proposition \ref{int rep cont}, we conclude.

\end{proof}

\textbf{Statement on conflicting interests:} The author declares that he has no conflicting interests.

\textbf{Acknowledgement:} The author wishes to thank the anonymous referee for his/her helpful remarks.

\bibliography{Flowsbib}
\bibliographystyle{abbrv} 
\end{document}